\documentclass[reqno]{amsart}

\usepackage{a4wide}
\usepackage[english]{babel}
\usepackage[utf8]{inputenc}
\usepackage{amsmath}
\usepackage{braket}
\usepackage[hidelinks]{hyperref}
\usepackage{tikz}

\renewcommand{\Re}{\operatorname{Re}}
\newcommand{\AC}{\operatorname{AC}}
\newcommand{\erf}{\operatorname{erf}}
\newcommand{\sgn}{\operatorname{sgn}}
\newcommand{\loc}{{\operatorname{loc}}}
\newcommand{\Arg}{\operatorname{Arg}}

\newtheorem{thm}{Theorem}[section]
\newtheorem{prop}[thm]{Proposition}
\newtheorem{cor}[thm]{Corollary}
\newtheorem{lem}[thm]{Lemma}
\newtheorem{defi}[thm]{Definition}
\newtheorem{ass}[thm]{Assumption}
\newtheorem{exam}[thm]{Example}
\newtheorem{rem}[thm]{Remark}
\numberwithin{equation}{section}

\title[Oscillatory integrals and the Schrödinger equation]{On a class of oscillatory integrals and their application to the time dependent Schrödinger equation}

\author[Jussi Behrndt]{J. Behrndt}
\address{(JB) Technische Universität Graz, Institut für Angewandte Mathematik, Steyrergasse 30, 8010 Graz, Austria}
\email{behrndt@tugraz.at}

\author[Peter Schlosser]{P. Schlosser}
\address{(PS) Politecnico di Milano, Dipartimento di Matematica, Via E. Bonardi 9, 20133 Milano, Italy}
\email{pschlosser@math.tugraz.at}

\begin{document}

\begin{abstract}
In this paper a class of oscillatory integrals is interpreted as a limit of Lebesgue integrals with Gaussian regularizers. The convergence of the regularized integrals is shown with an improved version of iterative integration by parts that generates additional decaying factors and hence leads to better integrability properties. The general abstract results are then applied to the Cauchy problem for the one dimensional time dependent Schrödinger equation, where the solution is expressed for $C^n$-regular initial conditions with polynomial growth at infinity via the Green's function as an oscillatory integral.
\end{abstract}

\maketitle

\section{Introduction}

In the analysis of the Cauchy problem for the one dimensional time dependent Schrödinger equation
\begin{equation}\label{Eq_Cauchy_intro}
\begin{split}
i\frac{\partial}{\partial t}\Psi(t,x)&=\Big(-\frac{\partial^2}{\partial x^2}+V(t,x)\Big)\Psi(t,x),\qquad t>0,\,x\in\mathbb{R}, \\
\Psi(0,x)&=F(x),\hspace{3.9cm} x\in\mathbb{R},
\end{split}
\end{equation}
with a given potential $V$ and initial condition $F$, one typically wants to express the solution $\Psi$ with the help of the corresponding {\it Green's function} $G(t,x,y)$ in the form
\begin{equation}\label{Eq_G_integral_intro}
\Psi(t,x)=\int_\mathbb{R}G(t,x,y)F(y)dy.
\end{equation}
To ensure the existence of the integral \eqref{Eq_G_integral_intro} in the usual Lebesgue sense one often imposes strong smoothness and decay assumptions on the initial condition $F$. In particular, test or Schwartz functions are a natural choice to avoid technical difficulties, see for example the constructions of explicit Green's functions in \cite{C09,C93,GS86,M89,S87}. However, if $F$ does not satisfy such additional requirements, it is necessary to specify in which sense the integral \eqref{Eq_G_integral_intro} has to be understood.  

The main objective of this paper is to propose a new constructive and explicit approach towards integrals that do not converge in the usual Lebesgue sense, but where the integrand has a fast oscillatory behaviour. Our technique  is inspired by the above Cauchy problem and its application in the theory of Aharonov-Berry superoscillations (a wave phenomenon, where low frequency waves interact almost destructively in such a way that the resulting wave has a very small amplitude but an arbitrary large frequency; cf. \cite{ADV88,B94} for the physical and mathematical origin of this effect, and \cite{B19} for a comprehensive survey), but will be developed in a general independent framework. To familiarize the reader with our method let us consider \eqref{Eq_Cauchy_intro} with the initial condition 
\begin{equation}\label{Eq_Plane_wave}
F(y)=e^{i\kappa y},
\end{equation}
which is a plane wave with frequency $\kappa\in\mathbb{R}$, the most important initial condition in the above mentioned application on superoscillations. Typically, the Green's function in \eqref{Eq_G_integral_intro} does not possess any decay as $y\to\pm\infty$ and thus the integral \eqref{Eq_G_integral_intro} with $F$ in \eqref{Eq_Plane_wave} is not absolutely convergent in the Lebesgue sense. In order to still give meaning to the integral, one may use that the Green's function of the Schrödinger equation often admits a decomposition of the form
\begin{equation*}
G(t,x,y)=e^{ia(t)y^2}\widehat{G}(t,x,y),
\end{equation*}
where the functions $a$ and $\widehat{G}$ satisfy certain additional conditions. Using this decomposition turns the integral \eqref{Eq_G_integral_intro} into
\begin{equation}\label{Eq_Ghat_integral_intro}
\Psi(t,x)=\int_\mathbb{R}e^{ia(t)y^2}\widehat{G}(t,x,y)F(y)dy.
\end{equation}
Now, one may interpret \eqref{Eq_Ghat_integral_intro} as a generalized Fresnel integral by formally substituting $y\to ye^{i\alpha}$ for some $\alpha>0$. This leads to a rotation of the integration path into the complex plane and turns the oscillating term $e^{ia(t)y^2}$ into the Gaussian decaying factor $e^{ia(t)(ye^{i\alpha})^2}$, which may lead to an absolutely convergent integral. This method typically requires that the integrand extends holomorphically into the complex plane and was employed in the context of time evolution of superoscillations recently in, e.g., \cite{ABCS22,ACSS19,S22}.

In the present paper we propose a different method, based on iterative integration by parts, to interpret \eqref{Eq_Ghat_integral_intro} as a limit of absolutely convergent integrals. This avoids the analyticity assumption on the integrand and replaces it by some $C^n$-regularity. Moreover, the integrand and hence the initial condition is allowed to grow polynomially at $y\to\pm\infty$, see Theorem~\ref{thm_Psi}. In order to sketch the key idea, we note that the fundamental structure of the integral \eqref{Eq_Ghat_integral_intro} is of the form
\begin{equation}\label{Eq_Integral_f_intro}
\int_b^\infty e^{iay^2}f(y)dy,
\end{equation}
where we only consider the integral along the interval $[b,\infty)$ with $b>0$ and $a\in\mathbb R\setminus\{0\}$. This integral can be rewritten by formally using integration by parts, as
\begin{equation}\label{Eq_Integration_by_parts_intro}
\int_b^\infty e^{iay^2}f(y)dy=\frac{-1}{2ia}\bigg(e^{iab^2}\frac{f(b)}{b}-\int_b^\infty e^{iay^2}\frac{f(y)}{y^2}dy+\int_b^\infty e^{iay^2}\frac{f'(y)}{y}dy\bigg),
\end{equation}
where we have assumed that the evaluation $e^{iay^2} f(y)/y$ for $y\to\infty$ vanishes. Thus, instead of the function $f$, it suffices to integrate the functions $f(y)/y^2$ and $f'(y)/y$, where the additional decaying factors $1/y^2$ and $1/y$ lead to better integrability property at $\infty$. This observation can be made rigorous by inserting the Gaussian factor $e^{-\varepsilon y^2}$, and when \eqref{Eq_Integration_by_parts_intro} is applied iteratively, it leads to the formula \eqref{Eq_Integral_formula}. It is then shown in Theorem~\ref{thm_Integral}, that for $n$-times continuously differentiable functions $f$ where the $n$-th derivative grows at most polynomially with order $\alpha<n-1$, we can consider the oscillatory integral as the limit of regularized integrals
\begin{equation}\label{Eq_Oscillatory_integral_intro}
\int_b^\infty e^{iay^2}f(y)dy:=\lim\limits_{\varepsilon\to 0^+}\int_b^\infty e^{-\varepsilon y^2}e^{iay^2}f(y)dy.
\end{equation}

In Section~\ref{sec_Time_dependent_Schroedinger_equation} this technique and the general results from Section~\ref{sec_Oscillatory_integrals} will be applied to the Cauchy problem for the one dimensional time dependent Schrödinger equation \eqref{Eq_Cauchy_intro}. In particular, assuming that the Green's function satisfies Assumption~\ref{ass_Greensfunction}, we conclude that for a class of $C^n$-regular initial conditions $F$ with polynomial growth, the solution of \eqref{Eq_Cauchy_intro} can be expressed as the limit 
\begin{equation}\label{Eq_Psi_intro}
\Psi(t,x)=\lim\limits_{\varepsilon\to 0^+}\int_\mathbb{R}e^{-\varepsilon y^2}G(t,x,y)F(y)dy,\qquad t >0,\,x\in\mathbb{R}.
\end{equation}
Here we shall also rely on the Leibniz rule for oscillatory integrals from Theorem~\ref{thm_Integral_derivative}, which in the concrete situation \eqref{Eq_Psi_intro} leads to
\begin{align*}
\frac{\partial^2}{\partial x^2}\Psi (t,x)&=\lim\limits_{\varepsilon\to 0^+}\int_\mathbb{R}e^{-\varepsilon y^2}\frac{\partial^2}{\partial x^2}G(t,x,y)F(y)dy,\\
\frac{\partial}{\partial t}\Psi (t,x)&=\lim\limits_{\varepsilon\to 0^+}\int_\mathbb{R}e^{-\varepsilon y^2}\frac{\partial}{\partial t}G(t,x,y)F(y)dy,
\end{align*}
and thus shows that $\Psi$ in \eqref{Eq_Psi_intro} is a solution of the time dependent Schrödinger equation. Furthermore, using the abstract continuity result Proposition~\ref{contiprop}, it turns out in Theorem~\ref{thm_Psi_continuous_dependency} that $\Psi$ depends continuously on the initial condition $F$. As a simple illustration of our general results we consider the free particle in Section~\ref{sec_The_free_particle}, where we also discuss the initial condition \eqref{Eq_Plane_wave} and compute the moments of the corresponding Green's function, i.e. choose the initial conditions $F(y)=y^m$. We also refer to \cite{ABCS22,ACSST13,ACSST17,ACSS19,BCSS14,CSSY22} for other explicit examples of Green's functions and related considerations in the context of superoscillations.

Finally, we briefly connect and relate our investigations to the general theory of {\it oscillatory integrals}, which appear in various branches in mathematics and physics. Note first that the single valued integral \eqref{Eq_Integral_f_intro} is an oscillatory integral of the form
\begin{equation}\label{Eq_Oscillatory_integral_first_kind}
\int_\Omega e^{ia\phi(y)}f(y)dy.
\end{equation}
While our method \eqref{Eq_Integration_by_parts_intro} of iterative integration by parts uses the term $e^{iay^2}$ to gain additional powers of $1/y$ in order to make the integral absolutely convergent, typically these integrals are treated for functions $f$ which are already absolutely integrable. The main interest is mostly on the asymptotic behaviour when $a\to\infty$, and the classical tool is the method of stationary phase; we refer the interested reader to the monographs by Hörmander \cite[Section~7.7]{H03}, Sogge \cite[Chapter~1]{S17} and Stein \cite[Chapter~VIII]{S93}. The case when $f$ does not possess any decay at infinity is treated in, e.g., \cite{AM04,AM05}, \cite[Section 1.2]{H71} and \cite[Section~0.5]{S17}. There, the oscillatory integral is, similar to \eqref{Eq_Oscillatory_integral_intro}, defined as the limit of regularized integrals
\begin{equation*}
\lim\limits_{\varepsilon\to 0^+}\int_\Omega\rho(\varepsilon y)e^{ia\phi(y)}f(y)dy,
\end{equation*}
where $\rho$ is a function with $\rho(0)=1$, which makes the above integral absolute convergent for every $\varepsilon>0$. However, the function $f$ still has to have $C^\infty$-regularity with growth assumptions on all derivatives, while in this paper we present an improved integration by parts method which reduces, at least for $\phi(y)=y^2$ and Gaussian regularizers $\rho(y)=e^{-y^2}$, the regularity assumptions to $C^n$, and polynomial boundedness is only needed for the $n$-th derivative. Typical examples are the Airy integral in \cite[Section 7.6]{H03} or the Stein-Wainger oscillatory integral treated in \cite{PP10,P08,RS87,SW70}.

Another possible approach to oscillatory integrals, which formally seems closer to the Green's function integral \eqref{Eq_Ghat_integral_intro}, is to introduce a second parameter $x$ as well as an integral kernel $K(x,y)$ into the integral \eqref{Eq_Oscillatory_integral_first_kind}. However, while we still consider this integral as the limit of regularized integrals \eqref{Eq_Psi_intro}, the classical approach in this case is to consider the integral
\begin{equation}\label{Eq_Oscillatory_integral_second_kind}
\int_\Omega e^{i\phi(x,y)}K(x,y)f(y)dy
\end{equation}
only for $f$ in a dense subset, and then extend it either as a bounded operator or in the sense of distributions. Mapping properties in between the Lebesgue spaces $L^p(\Omega)$ and $L^q(\Omega)$ are discussed in, e.g., \cite{B91,H73,L06,S16}. Such operators and their applications in the theory of partial differential equations were thoroughly investigated by Duistermaat and Hörmander in \cite{DH71,D74,H71,H03}. Typical examples for \eqref{Eq_Oscillatory_integral_second_kind} are the standard Fourier transform ($\phi(x,y)=xy$ and $K(x,y)=1$ on $\Omega=\mathbb R^n$), or more general Fourier integral operators, also with measure-valued kernels $K$ which restrict the integral to a submanifold \cite{BG11,S93,T03}. As another important application we point out that in the papers \cite{ABB96,AB93,AGM02,AH76,AM04} by Albeverio and collaborators, the infinite dimensional Feynman path integral is treated as the limit of finite dimensional oscillatory integrals. \medskip

\noindent {\bf Acknowledgements.} We are indebted to Jean-Claude Cuenin and Ahmed Sebbar for fruitful discussions and helpful remarks. This research was funded by the Austrian Science Fund (FWF) Grant-DOI: 10.55776/P33568 and 10.55776/J4685. The research of P.S. is also funded 
by the European Union--NextGenerationEU.

\section{Spaces of polynomially bounded $C^n$-functions}

In this preparatory section we introduce and study two families of spaces of $n$-times continuously differentiable functions, which both play an important role in this paper. First, in Definition~\ref{defi_Cnalpha} for $\alpha\geq 0$ and $b>0$ we define the space $C_\alpha^n([b,\infty))$, where only the $n$-th derivative is assumed to satisfy a polynomial bound $y^\alpha$ at $\infty$. For functions $f$ from this space we will define the oscillatory integral $\mathcal{I}_{a,b}(f)$ in Theorem~\ref{thm_Integral}. Second, we consider the space $C^n(\mathbb{R},r^\alpha)$ with $r(y)=1+\vert y\vert$ in Definition~\ref{defi_Cnr}, where all derivatives satisfy the polynomial bound $r(y)^\alpha$ at $\infty$. This space turns out to be more convenient when products of functions appear in the integrals, such as it is the case for the Green's function and the initial condition in Theorem~\ref{thm_Psi}.

\begin{defi}\label{defi_Cnalpha}
For $b>0$, $n\in\mathbb{N}_0$, and $\alpha\geq 0$, define the space
\begin{equation}\label{Eq_Cnalpha_space}
C^n_\alpha([b,\infty)):=\Set{f\in C^n([b,\infty)) | \vert f^{(n)}(y)\vert\leq My^\alpha,\;\text{for some}\;M\geq 0}
\end{equation}
equipped with the norm
\begin{equation}\label{Eq_Cnalpha_Norm}
\Vert f\Vert_{C_\alpha^n([b,\infty))}:=\sum\limits_{k=0}^{n-1}\frac{|f^{(k)}(b)|}{b^{n-k+\alpha}}+\sup\limits_{y\in[b,\infty)}\frac{|f^{(n)}(y)|}{y^\alpha};
\end{equation}
here (and in the following) we use the convention $\sum_{k=0}^{-1}:=0$.
\end{defi}

Observe that for $f\in C_\alpha^n([b,\infty))$ only a bound on the highest order derivative $f^{(n)}$ is required. However, the next Proposition~\ref{prop_f_derivatives_estimate} also provides bounds for the lower order derivatives, which will be used frequently in the following. The estimate \eqref{Eq_fk_estimate} is also useful to verify that $C_\alpha^n([b,\infty))$ is a Banach space; this can be done in the same way as for $C^n([b,\infty))$ and is not repeated here.

\begin{prop}\label{prop_f_derivatives_estimate}
Let $b>0$, $n\in\mathbb{N}_0$, and $\alpha\geq 0$. Then for any $f\in C^n_\alpha([b,\infty))$ one has 
\begin{equation}\label{Eq_fk_estimate}
|f^{(k)}(y)|\leq\Vert f\Vert_{C^n_\alpha([b,\infty))}y^{n-k+\alpha},\qquad y\in[b,\infty),\,k\in\{0,\dots,n\}. 
\end{equation}
\end{prop}

\begin{proof}
Let us set $M:=\sup_{y\in[b,\infty)}\frac{|f^{(n)}(y)|}{y^\alpha}$ and prove by induction for every $k\in\{0,\dots,n\}$ the stronger inequality
\begin{equation}\label{Eq_f_derivatives_estimate_1}
|f^{(k)}(y)|\leq\bigg(\sum\limits_{l=k}^{n-1}\frac{|f^{(l)}(b)|}{b^{n-l+\alpha}}+M\bigg)y^{n-k+\alpha},\qquad y\in[b,\infty).
\end{equation}
Since $f\in C^n_\alpha([b,\infty))$, the inequality \eqref{Eq_f_derivatives_estimate_1} for $k=n$ follows immediately from the definition of $M$. For the induction step $k\to k-1$, we get
\begin{align*}
|f^{(k-1)}(y)|&=\bigg|f^{(k-1)}(b)+\int_b^yf^{(k)}(z)dz\bigg| \\
&\leq|f^{(k-1)}(b)|+\bigg(\sum\limits_{l=k}^{n-1}\frac{|f^{(l)}(b)|}{b^{n-l+\alpha}}+M\bigg)\int_b^yz^{n-k+\alpha}dz \\
&\leq|f^{(k-1)}(b)|+\bigg(\sum\limits_{l=k}^{n-1}\frac{|f^{(l)}(b)|}{b^{n-l+\alpha}}+M\bigg)\frac{y^{n-k+\alpha+1}}{n-k+\alpha+1} \\
&\leq|f^{(k-1)}(b)|\Big(\frac{y}{b}\Big)^{n-k+\alpha+1}+\bigg(\sum\limits_{l=k}^{n-1}\frac{|f^{(l)}(b)|}{b^{n-l+\alpha}}+M\bigg)y^{n-k+\alpha+1} \\
&=\bigg(\sum\limits_{l=k-1}^{n-1}\frac{|f^{(l)}(b)|}{b^{n-l+\alpha}}+M\bigg)y^{n-k+\alpha+1},\qquad y\in[b,\infty).
\end{align*}
This shows estimate \eqref{Eq_f_derivatives_estimate_1}, which implies the inequality \eqref{Eq_fk_estimate}.
\end{proof}

In the following we collect four corollaries that are all based on the bound \eqref{Eq_fk_estimate}. The first one is a continuous embedding result.

\begin{cor}\label{cor_Reducing_interval}
Let $b>0$, $n\in\mathbb{N}_0$, $\alpha\geq 0$, and $f\in C_\alpha^n([b,\infty))$. Then for every $c\geq b$ one has $f|_{[c,\infty)}\in C_\alpha^n([c,\infty))$ and
\begin{equation*}
\Vert f|_{[c,\infty)}\Vert_{C_\alpha^n([c,\infty))}\leq(n+1)\Vert f\Vert_{C_\alpha^n([b,\infty))}.
\end{equation*}
\end{cor}

\begin{proof}
It is clear from the definition of the space \eqref{Eq_Cnalpha_space} that  $f|_{[c,\infty)}\in C_\alpha^n([c,\infty))$. Using \eqref{Eq_fk_estimate} for $y=c$, we can estimate the norm \eqref{Eq_Cnalpha_Norm} of the restricted function $f|_{[c,\infty)}$ by

\begin{align*}
\Vert f|_{[c,\infty)}\Vert_{C_\alpha^n([c,\infty))}&=\sum\limits_{k=0}^{n-1}\frac{|f^{(k)}(c)|}{c^{n-k+\alpha}}+\sup\limits_{y\in[c,\infty)}\frac{|f^{(n)}(y)|}{y^\alpha} \\
&\leq\sum\limits_{k=0}^{n-1}\frac{\Vert f\Vert_{C_\alpha^n([b,\infty))}c^{n-k+\alpha}}{c^{n-k+\alpha}}+\sup\limits_{y\in[b,\infty)}\frac{|f^{(n)}(y)|}{y^\alpha} \\
&\leq(n+1)\Vert f\Vert_{C_\alpha^n([b,\infty))}. \qedhere
\end{align*}
\end{proof}

\begin{cor}
Let $b>0$, $n\in\mathbb{N}_0$, $\alpha\geq 0$, and $f\in C_\alpha^n([b,\infty))$. Then for every $m\in\{0,\dots,n\}$ one has $f^{(m)}\in C_\alpha^{n-m}([b,\infty))$ and
\begin{equation*}
\Vert f^{(m)}\Vert_{C_\alpha^{n-m}([b,\infty))}\leq\Vert f\Vert_{C_\alpha^n([b,\infty))}.
\end{equation*}
\end{cor}

\begin{proof}
Using the inequality \eqref{Eq_fk_estimate}, the derivatives of $f^{(m)}$ can be estimated by
\begin{equation*}
\Big|\frac{d^k}{dy^k}f^{(m)}(y)\Big|=|f^{(m+k)}(y)|\leq\Vert f\Vert_{C_\alpha^n([b,\infty))}y^{n-m-k+\alpha},\qquad 
k\in\{0,\dots,n-m\}.
\end{equation*}
This estimate shows that $f^{(m)}\in C_\alpha^{n-m}([b,\infty))$, with norm bounded by
\begin{align*}
\Vert f^{(m)}\Vert_{C_\alpha^{n-m}([b,\infty))}&=\sum\limits_{k=0}^{n-m-1}\frac{|f^{(m+k)}(b)|}{b^{n-m-k+\alpha}}+\sup\limits_{y\in[b,\infty)}\frac{|f^{(m+n-m)}(y)|}{y^\alpha} \\
&=\sum\limits_{k=m}^{n-1}\frac{|f^{(k)}(b)|}{b^{n-k+\alpha}}+\sup\limits_{y\in[b,\infty)}\frac{|f^{(n)}(y)|}{y^\alpha}\\
&\leq\Vert f\Vert_{C_\alpha^n},
\end{align*}
where the first equation is the definition of the norm \eqref{Eq_Cnalpha_Norm}, in the second equation we substituted $k\to k-m$, and in the last inequality we added the missing terms $k=0,\dots,m-1$ in the sum.
\end{proof}

\begin{cor}
Let $b>0$, $n,m\in\mathbb{N}_0$, and $\alpha,\beta\geq 0$. If $n\geq m$ and $\beta\geq\alpha+n-m$, then one has
$C_\alpha^n([b,\infty))\subseteq C_\beta^m([b,\infty))$ and
\begin{equation*}
\Vert f\Vert_{C_\beta^m([b,\infty))}\leq\frac{m+1}{b^{\beta-\alpha-n+m}}\Vert f\Vert_{C_\alpha^n([b,\infty))},\qquad f\in C_\alpha^n([b,\infty)).
\end{equation*}
\end{cor}

\begin{proof}
Using the inequality \eqref{Eq_fk_estimate}, we get for every $f\in C_\alpha^n([b,\infty))$ and $k\in\{0,\dots,n\}$ the estimate
\begin{equation*}
|f^{(k)}(y)|\leq\Vert f\Vert_{C_\alpha^n([b,\infty))}y^{n-k+\alpha}\leq\Vert f\Vert_{C_\alpha^n([b,\infty))}\frac{y^{m-k+\beta}}{b^{\beta-\alpha-n+m}},
\end{equation*}
where in the second inequality $y\geq b$ and $\beta\geq\alpha+n-m$ was used. This estimate shows that $f\in C_\beta^m([b,\infty))$, with norm bounded by
\begin{align*}
\Vert f\Vert_{C_\beta^m([b,\infty))}&=\sum\limits_{k=0}^{m-1}\frac{|f^{(k)}(b)|}{b^{m-k+\beta}}+\sup\limits_{y\in[b,\infty)}\frac{|f^{(m)}(y)|}{y^\beta} \\
&\leq\Vert f\Vert_{C_\alpha^n([b,\infty))}\bigg(\sum\limits_{k=0}^{m-1}\frac{1}{b^{\beta-\alpha-n+m}}+\frac{1}{b^{\beta-\alpha-n+m}}\bigg) \\
&=\Vert f\Vert_{C_\alpha^n([b,\infty))}\frac{m+1}{b^{\beta-\alpha-n+m}}. \qedhere
\end{align*}
\end{proof}

\noindent The next result multiplies functions $f\in C_\alpha^n([b,\infty))$ with monomials $y^p$. In the following we shall use the convention $\prod_{j=0}^{-1}:=1$.

\begin{cor}\label{cor_ypf}
Let $b>0$, $n\in\mathbb{N}_0$, $\alpha\geq 0$ and $f\in C^n_\alpha([b,\infty))$. Then for every $p\geq 0$ one has $y^pf\in C_{\alpha+p}^n([b,\infty))$ and
\begin{equation*}
\Vert y^pf\Vert_{C_{\alpha+p}^n([b,\infty))}\leq(p+n+1)^n\Vert f\Vert_{C_\alpha^n([b,\infty))}.
\end{equation*}
\end{cor}

\begin{proof}
Using the inequality \eqref{Eq_fk_estimate}, we get for every $k\in\{0,\dots,n\}$ the estimate
\begin{equation}\label{Eq_ypf_estimate_1}
\begin{split}
\Big|\frac{d^k}{dy^k}(y^pf(y))\Big|&=\bigg|\sum\limits_{l=0}^k{k\choose l}\bigg(\prod\limits_{j=0}^{l-1}(p-j)\bigg)y^{p-l}f^{(k-l)}(y)\bigg|  \\
&\leq\Vert f\Vert_{C_\alpha^n([b,\infty))}\sum\limits_{l=0}^k{k\choose l}\bigg(\prod\limits_{j=0}^{l-1}|p-j|\bigg)y^{n-k+\alpha+p}. 
\end{split}
\end{equation}
Estimating $|p-j|\leq p+j\leq p+n-1$, using the largest $j$-value $j\leq l-1\leq k-1\leq n-1$, the inequality \eqref{Eq_ypf_estimate_1} becomes
\begin{align*}
\Big|\frac{d^k}{dy^k}(y^pf(y))\Big|&\leq\Vert f\Vert_{C_\alpha^n([b,\infty))}\sum\limits_{l=0}^k{k\choose l}(p+n-1)^ly^{n-k+\alpha+p} \\
&=\Vert f\Vert_{C_\alpha^n([b,\infty))}(p+n)^ky^{n-k+\alpha+p}
\end{align*}
for $y\in[b,\infty)$. In the second line we used the binomial formula $\sum_{l=0}^k{k\choose l}\xi^l=(\xi+1)^k$. This estimate shows that $y^pf\in C_{\alpha+p}^n([b,\infty))$, with norm bounded by
\begin{align*}
\Vert y^pf\Vert_{C_{\alpha+p}^n([b,\infty))}&=\sum\limits_{k=0}^{n-1}\frac{\Big|\frac{d^k}{dy^k}(y^pf(y))\big|_{y=b}\Big|}{b^{n-k+\alpha+p}}+\sup\limits_{y\in[b,\infty)}\frac{|\frac{d^n}{dy^n}(y^pf(y))|}{y^{\alpha+p}} \\
&\leq\Vert f\Vert_{C_\alpha^n([b,\infty))}\bigg(\sum\limits_{k=0}^{n-1}(p+n)^k+(p+n)^n\bigg) \\
&\leq\Vert f\Vert_{C_\alpha^n([b,\infty))}\sum\limits_{k=0}^n{n\choose k}(p+n)^k \\
&=\Vert f\Vert_{C_\alpha^n([b,\infty))}(p+n+1)^n,
\end{align*}
where in the last line we again used the binomial formula.
\end{proof}

The next remark shows that the function spaces $C_\alpha^n$ in Definition~\ref{defi_Cnalpha} do not possess natural multiplication properties. The main reason for this is that the growth $y^\alpha$ in \eqref{Eq_Cnalpha_space} is only required for the highest derivative, while the lower derivatives may grow faster, see Proposition~\ref{prop_f_derivatives_estimate}. 

\begin{rem}
Observe that for $f\in C_\alpha^n([b,\infty))$ and $g\in C_\beta^n([b,\infty))$, $\alpha,\beta\geq 0$,  one can not conclude $fg\in C_{\alpha+\beta}^n([b,\infty))$ in general. E.g., for $f(y)=e^{iy}\in C_0^n([b,\infty))$ and $g(y)=y^{n+\beta}\in C^n_\beta([b,\infty))$, Corollary~\ref{cor_ypf} only ensures $fg\in C_{n+\beta}^n([b,\infty))$, but since
\begin{equation*}
\frac{d^n}{dy^n}(y^{n+\beta}e^{iy})=\sum\limits_{k=0}^n{n\choose k}\bigg(\prod\limits_{l=0}^{k-1}(n+\beta-l)\bigg)y^{n-k+\beta}i^{n-k}e^{iy}=\mathcal{O}(y^{n+\beta}),\quad\text{as }y\to\infty,
\end{equation*}
it is clear that $fg\notin C_\beta^n([b,\infty))$ for $n\neq 0$.
\end{rem}

Next we introduce another family of weighted $C^n$-spaces, which are more natural for treating products and hence will be used in Section~\ref{sec_Time_dependent_Schroedinger_equation}, where the product of Green's function and an initial condition appears in the integrand. As weights we consider powers of
\begin{equation*}
r(y):=1+|y|,
\end{equation*}
and in contrast to the $C_\alpha^n$-spaces, not only the $n$-th derivative, but also the lower order derivatives have to satisfy the same upper bound. However, in Lemma~\ref{lem_Cnr_Cnalpha} it turns out that (modulo restrictions) the spaces $C^n(\mathbb{R},r^\alpha)$ can be viewed as a subspace of $C_\alpha^n([b,\infty))$.

\begin{defi}\label{defi_Cnr}
For $n\in\mathbb{N}_0$ and $\alpha\geq 0$, define the space
\begin{equation*}
C^n(\mathbb{R},r^\alpha):=\Set{f\in C^n(\mathbb{R}) | \exists M\geq 0: |f^{(k)}(y)|\leq Mr(y)^\alpha,\,k\in\{0,\dots,n\}},
\end{equation*}
equipped with the norm
\begin{equation}\label{Eq_Cnr_norm}
\Vert f\Vert_{C^n(\mathbb{R},r^\alpha)}:=\max\limits_{k\in\{0,\dots,n\}}\sup_{y\in\mathbb{R}}\frac{|f^{(k)}(y)|}{r(y)^\alpha}.
\end{equation}
\end{defi}

We remark that the space $C^n(\mathbb{R},r^\alpha)$ is a Banach space, which can be proven in the same way as for the standard $C^n$-spaces and is not repeated here. Next we collect some useful features of the functions in $C^n(\mathbb{R},r^\alpha)$ and start with a natural product property.

\begin{prop}\label{prop_Cn_product}
Let $n\in\mathbb{N}_0$ and $\alpha,\beta\geq 0$. Then for every $f\in C^n(\mathbb{R},r^\alpha)$ and $g\in C^n(\mathbb{R},r^\beta)$ one has $fg\in C^n(\mathbb{R},r^{\alpha+\beta})$ and
\begin{equation*}
\Vert fg\Vert_{C^n(\mathbb{R},r^{\alpha+\beta})}\leq 2^n\Vert f\Vert_{C^n(\mathbb{R},r^\alpha)}\Vert g\Vert_{C^n(\mathbb{R},r^\beta)}.
\end{equation*}
\end{prop}

\begin{proof}
For every $k\in\{0,\dots,n\}$ the $k$-th derivative of the product $fg$ can be estimated by
\begin{align*}
\Big|\frac{d^k}{dy^k}(f(y)g(y))\Big|&=\Big|\sum\limits_{l=0}^k{k\choose l}f^{(l)}(y)g^{(k-l)}(y)\Big| \\
&\leq\sum\limits_{l=0}^k{k\choose l}\Vert f\Vert_{C^n(\mathbb{R},r^\alpha)}\Vert g\Vert_{C^n(\mathbb{R},r^\beta)}r(y)^{\alpha+\beta} \\
&=2^k\Vert f\Vert_{C^n(\mathbb{R},r^\alpha)}\Vert g\Vert_{C^n(\mathbb{R},r^\beta)}r(y)^{\alpha+\beta}.
\end{align*}
This estimate shows that $fg\in C^n(\mathbb{R},r^{\alpha+\beta})$, with norm bounded by
\begin{equation*}
\Vert fg\Vert_{C^n(\mathbb{R},r^{\alpha+\beta})}\leq\max_{k\in\{0,\dots,n\}}2^k\Vert f\Vert_{C^n(\mathbb{R},r^\alpha)}\Vert g\Vert_{C^n(\mathbb{R},r^\beta)}=2^n\Vert f\Vert_{C^n(\mathbb{R},r^\alpha)}\Vert g\Vert_{C^n(\mathbb{R},r^\beta)}. \qedhere
\end{equation*}
\end{proof}

\begin{lem}\label{lem_ymf_expf}
Let $n\in\mathbb{N}_0$, $\alpha\geq 0$, and $f\in C^n(\mathbb{R},r^\alpha)$. Then the following assertions hold.

\begin{enumerate}
\item[(i)] For every $m\in\mathbb{N}_0$ one has $y^mf\in C^n(\mathbb{R},r^{\alpha+m})$ and
\begin{equation}\label{Eq_ymf_norm_estimate}
\Vert y^mf\Vert_{C^n(\mathbb{R},r^{\alpha+m})}\leq(1+m)^n\Vert f\Vert_{C^n(\mathbb{R},r^\alpha)}.
\end{equation}

\item[(ii)] For every $\kappa\in\mathbb{R}$ one has $e^{i\kappa y}f\in C^n(\mathbb{R},r^\alpha)$ and
\begin{equation}\label{Eq_expf_norm_estimate}
\Vert e^{i\kappa y}f\Vert_{C^n(\mathbb{R},r^\alpha)}\leq(1+|\kappa|)^n\Vert f\Vert_{C^n(\mathbb{R},r^\alpha)}.
\end{equation}
\end{enumerate}
\end{lem}

\begin{proof}
(i)\;\;For every $k\in\{0,\dots,n\}$ the $k$-th derivative of $y^mf(y)$ can be estimated by
\begin{align*}
\Big|\frac{d^k}{dy^k}(y^mf(y))\Big|&=\bigg|\sum\limits_{l=0}^{\min\{k,m\}}{k\choose l}\frac{m!}{(m-l)!}y^{m-l}f^{(k-l)}(y)\bigg|\\
&\leq\sum\limits_{l=0}^{\min\{k,m\}}{k\choose l}m^l\Vert f\Vert_{C^n(\mathbb{R},r^\alpha)}r(y)^{\alpha+m-l} \\
&\leq\sum\limits_{l=0}^k{k\choose l}m^l\Vert f\Vert_{C^n(\mathbb{R},r^\alpha)}r(y)^{\alpha+m}\\
&=(1+m)^k\Vert f\Vert_{C^n(\mathbb{R},r^\alpha)}r(y)^{\alpha+m}.
\end{align*}
This inequality shows that $y^mf\in C^n(\mathbb{R},r^{\alpha+m})$, with norm bounded by \eqref{Eq_ymf_norm_estimate}. \medskip

\noindent (ii)\;\;For every $k\in\{0,\dots,n\}$ the $k$-th derivative of $e^{i\kappa y}f(y)$ can be estimated by
\begin{align*}
\Big|\frac{d^k}{dy^k}(e^{i\kappa y}f(y))\Big|&=\bigg|\sum\limits_{l=0}^k{k\choose l}(i\kappa)^le^{i\kappa y}f^{(k-l)}(y)\bigg| \\
&\leq\sum\limits_{l=0}^k{k\choose l}|\kappa|^l\Vert f\Vert_{C^n(\mathbb{R},r^\alpha)}r(y)^\alpha\\
&=(1+|\kappa|)^k\Vert f\Vert_{C^n(\mathbb{R},r^\alpha)}r(y)^\alpha.
\end{align*}
This inequality shows that $e^{i\kappa y}f\in C^n(\mathbb{R},r^\alpha)$, with norm bounded by \eqref{Eq_expf_norm_estimate}.
\end{proof}

\noindent Next, we show that the space $C^n(\mathbb{R},r^\alpha)$ is invariant with respect to the shift of the function.

\begin{lem}\label{lem_Cn_shifted}
Let $n\in\mathbb{N}_0$, $\alpha\geq 0$, and $f\in C^n(\mathbb{R},r^\alpha)$. Then for every fixed $x\in\mathbb{R}$ one has $f(\,\cdot\,+x)\in C^n(\mathbb{R},r^\alpha)$ and 
\begin{equation*}
\Vert f(\,\cdot\,+x)\Vert_{C^n(\mathbb{R},r^\alpha)}\leq(1+|x|)^\alpha\Vert f\Vert_{C^n(\mathbb{R},r^\alpha)}.
\end{equation*}
\end{lem}

\begin{proof}
For every $k\in\{0,\dots,n\}$ the $k$-th derivative of $f(\,\cdot\,+x)$ can be estimated by
\begin{equation*}
\Big|\frac{d^k}{dy^k}f(y+x)\Big|=|f^{(k)}(y+x)|\leq\Vert f\Vert_{C^n(\mathbb{R},r^\alpha)}(1+|y+x|)^\alpha\leq\Vert f\Vert_{C^n(\mathbb{R},r^\alpha)}(1+|x|)^\alpha(1+|y|)^\alpha.
\end{equation*}
This estimate shows that $f(\,\cdot\,+x)\in C^n(\mathbb{R},r^\alpha)$, with norm bounded by
\begin{equation*}
\Vert f(\,\cdot\,+x)\Vert_{C^n(\mathbb{R},r^\alpha)}\leq\Vert f\Vert_{C^n(\mathbb{R},r^\alpha)}(1+|x|)^\alpha. \qedhere
\end{equation*}
\end{proof}

The following lemma shows, that the function space $C_\alpha^n([b,\infty))$ from Definition~\ref{defi_Cnalpha} is indeed in some sense larger then the weighted $C^n$-spaces $C^n(\mathbb{R},r^\alpha)$ from Definition~\ref{defi_Cnr}.

\begin{lem}\label{lem_Cnr_Cnalpha}
Let $n\in\mathbb{N}_0$, $\alpha\geq 0$, and $f\in C^n(\mathbb{R},r^\alpha)$. Then for $b>0$ one has $f|_{[b,\infty)}\in C^n_\alpha([b,\infty))$ and
\begin{equation*}
\Vert f|_{[b,\infty)}\Vert_{C^n_\alpha([b,\infty))}\leq\Big(1+\frac{1}{b}\Big)^{n+\alpha}\Vert f\Vert_{C^n(\mathbb{R},r^\alpha)}.
\end{equation*}
\end{lem}

\begin{proof}
For every $k\in\{0,\dots,n\}$ the $k$-th derivative of $f$ can be estimated by
\begin{equation*}
|f^{(k)}(y)|\leq\Vert f\Vert_{C^n(\mathbb{R},r^\alpha)}(1+y)^\alpha\leq\Vert f\Vert_{C^n(\mathbb{R},r^\alpha)}\Big(1+\frac{1}{b}\Big)^\alpha y^\alpha,\qquad y\in[b,\infty),
\end{equation*}
where $1+y\leq(1+\frac{1}{b})y$ was used for every $y\geq b$. This estimate shows that $f|_{[b,\infty)}\in C^n_\alpha([b,\infty))$, with norm bounded by
\begin{align*}
\Vert f|_{[b,\infty)}\Vert_{C_\alpha^n([b,\infty))}&=\sum\limits_{k=0}^{n-1}\frac{|f^{(k)}(b)|}{b^{n-k+\alpha}}+\sup\limits_{y\in[b,\infty)}\frac{|f^{(n)}(y)|}{y^\alpha} \\
&\leq\Vert f\Vert_{C^n(\mathbb{R},r^\alpha)}\Big(1+\frac{1}{b}\Big)^\alpha\bigg(\sum\limits_{k=0}^{n-1}\frac{1}{b^{n-k}}+1\bigg)\\
&\leq\Vert f\Vert_{C^n(\mathbb{R},r^\alpha)}\Big(1+\frac{1}{b}\Big)^{n+\alpha},
\end{align*}
where in the last inequality we used
\begin{equation*}
\sum\limits_{k=0}^{n-1}\frac{1}{b^{n-k}}+1=\sum_{k=0}^n\frac{1}{b^{n-k}}\leq\sum_{k=0}^n{n\choose k}\frac{1}{b^{n-k}}=\Big(1+\frac{1}{b}\Big)^n. \qedhere
\end{equation*}
\end{proof}

\section{Oscillatory integrals}\label{sec_Oscillatory_integrals}

\noindent In this section we want to give meaning to oscillatory integrals of the form
\begin{equation}\label{Eq_Integral_not_L1}
\int_b^\infty e^{iay^2}f(y)dy,
\end{equation}
where $a\in\mathbb{R}\setminus\{0\}$, $b>0$, and the function $f:[b,\infty)\to\mathbb{C}$ is not necessarily (absolutely) integrable on $[b,\infty)$. The idea is to insert the Gaussian regularizer $e^{-\varepsilon(y-y_0)^2}$, $\varepsilon>0$, $y_0\in\mathbb{R}$, into \eqref{Eq_Integral_not_L1} and consider the regularized Lebesgue integral
\begin{equation}\label{Eq_Integral_epsilon}
I_{a,b}^{\varepsilon,y_0}(f):=\int_b^\infty e^{-\varepsilon(y-y_0)^2}e^{iay^2}f(y)dy.
\end{equation}
The oscillatory integral \eqref{Eq_Integral_not_L1} can now be defined as the limit $\varepsilon\to 0^+$ of the absolutely convergent integrals \eqref{Eq_Integral_epsilon}
\begin{equation}\label{Eq_Iab_formal}
\mathcal{I}_{a,b}(f):=\lim\limits_{\varepsilon\to 0^+}I_{a,b}^{\varepsilon,y_0}(f)=\lim\limits_{\varepsilon\to 0^+}\int_b^\infty e^{-\varepsilon(y-y_0)^2}e^{iay^2}f(y)dy,
\end{equation}
for functions $f$ for which this limit exists and is independent of the chosen $y_0\in\mathbb{R}$. We emphasize that in the following the notation $I_{a,b}^{\varepsilon,y_0}(f)$ is used for the standard Lebesgue integral \eqref{Eq_Integral_epsilon}, whereas we shall use the calligraphic symbol $\mathcal{I}_{a,b}(f)$ for the limit in \eqref{Eq_Iab_formal} which, in general, does not exist as an absolute convergent Lebesgue integral.

It will be shown in Theorem~\ref{thm_Integral} that for functions in the space $C_\alpha^n([b,\infty))$ of Definition~\ref{defi_Cnalpha} with $n>\alpha+1$, the limit \eqref{Eq_Iab_formal} exists. The key idea of the convergence of these oscillatory integrals is presented in the following Lemma~\ref{lem_Integration_by_parts}. Using integration by parts we will derive the identity \eqref{Eq_Integration_by_parts} for the regularized integral $I_{a,b}^{\varepsilon,y_0}$, which then will serve as a first step for the more complicated formula \eqref{Eq_Integral_formula}. We also refer to, e.g., \cite[Theorem 7.7.1]{H03}, \cite[Lemma 0.4.7, Lemma 1.1.2]{S17}, and \cite[Proposition 1 on page 331]{S93}, where similar iterative integration by parts techniques were used.

\begin{lem}\label{lem_Integration_by_parts}
Let $a\in\mathbb{R}\setminus\{0\}$, $b>0$, $\varepsilon>0$, and $y_0\in\mathbb{R}$. Then for every $f\in C^1_\alpha([b,\infty))$, $\alpha\geq 0$, and $\kappa\in\mathbb{N}_0$ 
one has 
\begin{equation}\label{Eq_Integration_by_parts}
I_{a,b}^{\varepsilon,y_0}\Big(\frac{f}{y_\varepsilon^\kappa}\Big)=\frac{-1}{2(ia-\varepsilon)}\bigg(\frac{e^{iab^2-\varepsilon(b-y_0)^2}f(b)}{b_\varepsilon^{\kappa+1}}-(\kappa+1)I_{a,b}^{\varepsilon,y_0}\Big(\frac{f}{y_\varepsilon^{\kappa+2}}\Big)+I_{a,b}^{\varepsilon,y_0}\Big(\frac{f'}{y_\varepsilon^{\kappa+1}}\Big)\bigg),
\end{equation}
where  $y_\varepsilon:=y+\frac{\varepsilon y_0}{ia-\varepsilon}$ and $b_\varepsilon:=b+\frac{\varepsilon y_0}{ia-\varepsilon}$.
\end{lem}

Formula \eqref{Eq_Integration_by_parts} shows, that instead of $\frac{f}{y_\varepsilon^\kappa}$, one can rather integrate the functions $\frac{f}{y_\varepsilon^{\kappa+2}}$ and $\frac{f'}{y_\varepsilon^{\kappa+1}}$. This procedure is illustrated in the following diagram. \medskip

\begin{center}
\begin{tikzpicture}
\draw (0,0) node[anchor=center] {$I_{a,b}^{\varepsilon,y_0}\big(\frac{f}{y_\varepsilon^\kappa}\big)$};
\draw[thick,->] (0,-0.3)--(0,-0.7);
\draw (0,-1.1) node[anchor=center] {$I_{a,b}^{\varepsilon,y_0}\big(\frac{f}{y_\varepsilon^{\kappa+2}}\big)$};
\draw[thick,->] (1,0)--(1.5,0);
\draw (2.7,0) node[anchor=center] {$I_{a,b}^{\varepsilon,y_0}\big(\frac{f'}{y_\varepsilon^{\kappa+1}}\big)$};
\end{tikzpicture}
\end{center}

\begin{proof}[Proof of Lemma~\ref{lem_Integration_by_parts}]
Using integration by parts we can rewrite the integral \eqref{Eq_Integral_epsilon} as
\begin{align}
I_{a,b}^{\varepsilon,y_0}\Big(\frac{f}{y_\varepsilon^\kappa}\Big)&=\int_b^\infty e^{iay^2-\varepsilon(y-y_0)^2}\frac{f(y)}{y_\varepsilon^\kappa}dy \notag \\
&=\frac{1}{2(ia-\varepsilon)}\int_b^\infty\frac{d}{dy}\big(e^{iay^2-\varepsilon(y-y_0)^2}\big)\frac{f(y)}{y_\varepsilon^{\kappa+1}}dy \label{Eq_Integration_by_parts_1} \\
&=\frac{1}{2(ia-\varepsilon)}\bigg(e^{iay^2-\varepsilon(y-y_0)^2}\frac{f(y)}{y_\varepsilon^{\kappa+1}}\bigg|_{y=b}^\infty-\int_b^\infty e^{iay^2-\varepsilon(y-y_0)^2}\frac{d}{dy}\Big(\frac{f(y)}{y_\varepsilon^{\kappa+1}}\Big)dy\bigg). \notag
\end{align}
Since $f\in C_\alpha^1([b,\infty))$ the boundary term satisfies
\begin{equation*}
\Big|e^{iay^2-\varepsilon(y-y_0)^2}\frac{f(y)}{y_\varepsilon^{\kappa+1}}\Big|\leq e^{-\varepsilon(y-y_0)^2}\frac{\Vert f\Vert_{C_\alpha^1([b,\infty))}y^{1+\alpha}}{|y_\varepsilon|^{\kappa+1}}\to 0\quad\text{as}\,\,y\to\infty,
\end{equation*}
where also \eqref{Eq_fk_estimate} was used. Therefore, \eqref{Eq_Integration_by_parts_1} simplifies to 
\begin{align*}
I_{a,b}^{\varepsilon,y_0}\Big(\frac{f}{y_\varepsilon^\kappa}\Big)&=\frac{-1}{2(ia-\varepsilon)}\bigg(e^{iab^2-\varepsilon(b-y_0)^2}\frac{f(b)}{b_\varepsilon^{\kappa+1}}+\int_b^\infty e^{iay^2-\varepsilon(y-y_0)^2}\frac{d}{dy}\Big(\frac{f(y)}{y_\varepsilon^{\kappa+1}}\Big)dy\bigg) \\
&=\frac{-1}{2(ia-\varepsilon)}\bigg(e^{iab^2-\varepsilon(b-y_0)^2}\frac{f(b)}{b_\varepsilon^{\kappa+1}}-(\kappa+1)I_{a,b}^{\varepsilon,y_0}\Big(\frac{f}{y_\varepsilon^{\kappa+2}}\Big)+I_{a,b}^{\varepsilon,y_0}\Big(\frac{f'}{y_\varepsilon^{\kappa+1}}\Big)\bigg). \qedhere
\end{align*}
\end{proof}

Next, we repeatedly apply \eqref{Eq_Integration_by_parts} to gain multiple powers of $\frac{1}{y_\varepsilon}$, which serve as regularizers of the integrand at infinity. The formula \eqref{Eq_Integral_formula} below is the crucial ingredient in Theorem~\ref{thm_Integral} to define the oscillatory integral \eqref{Eq_Integral_not_L1} as a limit via \eqref{Eq_Iab_formal}. The following diagram provides a schematic illustration of how \eqref{Eq_Integration_by_parts} is repeatedly applied: Instead of the integral $I_{a,b}^{\varepsilon,y_0}(f)$ in the top left corner, one can rather compute the $n$ integrals in the bottom row and the $m+1$ integrals in the right column. \medskip

\begin{center}
\begin{tikzpicture}
\draw (0,0) node[anchor=center] {$I_{a,b}^{\varepsilon,y_0}(f)$};
\draw[thick,->] (0,-0.3)--(0,-0.7);
\draw (0,-1.1) node[anchor=center] {$I_{a,b}^{\varepsilon,y_0}\big(\frac{f}{y_\varepsilon^2}\big)$};
\draw[thick,->] (0,-1.4)--(0,-1.8);
\draw (0,-2) node[anchor=center] {$\vdots$};
\draw[thick,->] (0,-2.4)--(0,-2.8);
\draw (0,-3.2) node[anchor=center] {$I_{a,b}^{\varepsilon,y_0}\big(\frac{f}{y_\varepsilon^{2m}}\big)$};
\draw[thick,->] (0,-3.55)--(0,-3.95);
\draw (0,-4.3) node[anchor=center] {$I_{a,b}^{\varepsilon,y_0}\big(\frac{f}{y_\varepsilon^{2+2m}}\big)$};
\draw[thick,->] (1,0)--(1.5,0);
\draw[thick,->] (1,-1.1)--(1.5,-1.1);
\draw[thick,->] (1,-3.2)--(1.5,-3.2);
\draw (2.7,0) node[anchor=center] {$I_{a,b}^{\varepsilon,y_0}\big(\frac{f'}{y_\varepsilon}\big)$};
\draw[thick,->] (2.7,-0.3)--(2.7,-0.7);
\draw (2.7,-1.1) node[anchor=center] {$I_{a,b}^{\varepsilon,y_0}\big(\frac{f'}{y_\varepsilon^3}\big)$};
\draw[thick,->] (2.7,-1.4)--(2.7,-1.8);
\draw (2.7,-2) node[anchor=center] {$\vdots$};
\draw[thick,->] (2.7,-2.4)--(2.7,-2.8);
\draw (2.7,-3.2) node[anchor=center] {$I_{a,b}^{\varepsilon,y_0}\big(\frac{f'}{y_\varepsilon^{1+2m}}\big)$};
\draw[thick,->] (2.7,-3.55)--(2.7,-3.95);
\draw (2.7,-4.3) node[anchor=center] {$I_{a,b}^{\varepsilon,y_0}\big(\frac{f'}{y_\varepsilon^{3+2m}}\big)$};
\draw[thick,->] (3.9,0)--(4.4,0);
\draw[thick,->] (3.9,-1.1)--(4.4,-1.1);
\draw[thick,->] (3.9,-3.2)--(4.4,-3.2);
\draw (4.8,0) node[anchor=center] {$\dots$};
\draw (4.8,-1.1) node[anchor=center] {$\dots$};
\draw (4.8,-3.2) node[anchor=center] {$\dots$};
\draw (4.8,-4.3) node[anchor=center] {$\dots$};
\draw[thick,->] (5.1,0)--(5.6,0);
\draw[thick,->] (5.1,-1.1)--(5.6,-1.1);
\draw[thick,->] (5.1,-3.2)--(5.6,-3.2);
\draw (6.95,0) node[anchor=center] {$I_{a,b}^{\varepsilon,y_0}\big(\frac{f^{(n-1)}}{y_\varepsilon^{n-1}}\big)$};
\draw[thick,->] (6.95,-0.35)--(6.95,-0.7);
\draw (6.95,-1.1) node[anchor=center] {$I_{a,b}^{\varepsilon,y_0}\big(\frac{f^{(n-1)}}{y_\varepsilon^{n+1}}\big)$};
\draw[thick,->] (6.95,-1.45)--(6.95,-1.8);
\draw (6.95,-2) node[anchor=center] {$\vdots$};
\draw[thick,->] (6.95,-2.4)--(6.95,-2.8);
\draw (6.95,-3.2) node[anchor=center] {$I_{a,b}^{\varepsilon,y_0}\big(\frac{f^{(n-1)}}{y_\varepsilon^{n-1+2m}}\big)$};
\draw[thick,->] (6.95,-3.55)--(6.95,-3.95);
\draw (6.95,-4.3) node[anchor=center] {$I_{a,b}^{\varepsilon,y_0}\big(\frac{f^{(n-1)}}{y_\varepsilon^{n+1+2m}}\big)$};
\draw[thick,->] (8.25,0)--(8.75,0);
\draw[thick,->] (8.25,-1.1)--(8.75,-1.1);
\draw[thick,->] (8.25,-3.2)--(8.75,-3.2);
\draw (9.95,0) node[anchor=center] {$I_{a,b}^{\varepsilon,y_0}\big(\frac{f^{(n)}}{y_\varepsilon^n}\big)$};
\draw (9.95,-1.1) node[anchor=center] {$I_{a,b}^{\varepsilon,y_0}\big(\frac{f^{(n)}}{y_\varepsilon^{n+2}}\big)$};
\draw (9.95,-2) node[anchor=center] {$\vdots$};
\draw (9.95,-3.2) node[anchor=center] {$I_{a,b}^{\varepsilon,y_0}\big(\frac{f^{(n)}}{y_\varepsilon^{n+2m}}\big)$};
\end{tikzpicture}
\end{center}

\begin{lem}
Let $a\in\mathbb{R}\setminus\{0\}$, $b>0$, $\varepsilon>0$, $y_0\in\mathbb{R}$, $m\in\mathbb{N}_0$, and $n\in\mathbb{N}$. Then for every $f\in C^n_\alpha([b,\infty))$, $\alpha\geq 0$, one has 
\begin{equation}\label{Eq_Integral_formula}
\begin{split}
I_{a,b}^{\varepsilon,y_0}(f)&=\sum\limits_{k=0}^{n-1}\sum\limits_{l=0}^m\frac{C_{k,l}e^{iab^2-\varepsilon(b-y_0)^2}f^{(k)}(b)}{(ia-\varepsilon)^{k+1+l}b_\varepsilon^{k+1+2l}}+\sum\limits_{l=0}^m\frac{C_{n-1,l}}{(ia-\varepsilon)^{n+l}}I_{a,b}^{\varepsilon,y_0}\Big(\frac{f^{(n)}}{y_\varepsilon^{n+2l}}\Big)  \\
&\quad-\sum\limits_{k=0}^{n-1}\frac{(k+1+2m)C_{k,m}}{(ia-\varepsilon)^{k+1+m}}I_{a,b}^{\varepsilon,y_0}\Big(\frac{f^{(k)}}{y_\varepsilon^{k+2+2m}}\Big), 
\end{split}
\end{equation}
where $y_\varepsilon:=y+\frac{\varepsilon y_0}{ia-\varepsilon}$ and $b_\varepsilon:=b+\frac{\varepsilon y_0}{ia-\varepsilon}$, and the coefficients $C_{k,l}\in\mathbb{R}$ are recursively given by
\begin{equation}\label{Eq_Coefficients}
C_{k,0}:=\Big(-\frac{1}{2}\Big)^{k+1}\qquad\text{and}\qquad C_{k,l+1}:=\sum\limits_{i=0}^k\frac{(-1)^{i+k}(i+1+2l)}{2^{k+1-i}}C_{i,l}.
\end{equation}
\end{lem}

\begin{proof}
{\it Step 1.} We fix $m=0$ and prove the formula \eqref{Eq_Integral_formula} inductively for $n\in\mathbb{N}$. For the induction start $n=1$, the formula \eqref{Eq_Integral_formula} reduces to
\begin{equation*}
I_{a,b}^{\varepsilon,y_0}(f)=e^{iab^2-\varepsilon(b-y_0)^2}\frac{C_{0,0}f(b)}{(ia-\varepsilon)b_\varepsilon}+\frac{C_{0,0}}{ia-\varepsilon}I_{a,b}^{\varepsilon,y_0}\Big(\frac{f'}{y_\varepsilon}\Big)-\frac{C_{0,0}}{ia-\varepsilon}I_{a,b}^{\varepsilon,y_0}\Big(\frac{f}{y_\varepsilon^2}\Big),
\end{equation*}
and holds by \eqref{Eq_Integration_by_parts} with $\kappa=0$ and the choice $C_{0,0}=-\frac{1}{2}$. For the induction step $n\to n+1$, we know by \eqref{Eq_Integral_formula} with $m=0$, that
\begin{equation*}
I_{a,b}^{\varepsilon,y_0}(f)=\sum\limits_{k=0}^{n-1}\frac{C_{k,0}e^{iab^2-\varepsilon(b-y_0)^2}f^{(k)}(b)}{(ia-\varepsilon)^{k+1}b_\varepsilon^{k+1}}+\frac{C_{n-1,0}}{(ia-\varepsilon)^n}I_{a,b}^{\varepsilon,y_0}\Big(\frac{f^{(n)}}{y_\varepsilon^n}\Big)-\sum\limits_{k=0}^{n-1}\frac{(k+1)C_{k,0}}{(ia-\varepsilon)^{k+1}}I_{a,b}^{\varepsilon,y_0}\Big(\frac{f^{(k)}}{y_\varepsilon^{k+2}}\Big).
\end{equation*}
Using \eqref{Eq_Integration_by_parts} with $f^{(n)}$ and $\kappa=n$, turns this formula into
\begin{align*}
I_{a,b}^{\varepsilon,y_0}(f)&=\sum\limits_{k=0}^{n-1}\frac{C_{k,0}e^{iab^2-\varepsilon(b-y_0)^2}f^{(k)}(b)}{(ia-\varepsilon)^{k+1}b_\varepsilon^{k+1}}-\sum\limits_{k=0}^{n-1}\frac{(k+1)C_{k,0}}{(ia-\varepsilon)^{k+1}}I_{a,b}^{\varepsilon,y_0}\Big(\frac{f^{(k)}}{y_\varepsilon^{k+2}}\Big) \\
&\quad-\frac{C_{n-1,0}}{2(ia-\varepsilon)^{n+1}}\bigg(\frac{e^{iab^2-\varepsilon(b-y_0)^2}f^{(n)}(b)}{b_\varepsilon^{n+1}}-(n+1)I_{a,b}^{\varepsilon,y_0}\Big(\frac{f^{(n)}}{y_\varepsilon^{n+2}}\Big)+I_{a,b}^{\varepsilon,y_0}\Big(\frac{f^{(n+1)}}{y_\varepsilon^{n+1}}\Big)\bigg) \\
&=\sum\limits_{k=0}^n\frac{C_{k,0}e^{iab^2-\varepsilon(b-y_0)^2}f^{(k)}(b)}{(ia-\varepsilon)^{k+1}b_\varepsilon^{k+1}}+\frac{C_{n,0}}{(ia-\varepsilon)^{n+1}}I_{a,b}^{\varepsilon,y_0}\Big(\frac{f^{(n+1)}}{y_\varepsilon^{n+1}}\Big) \\
&\quad-\sum\limits_{k=0}^n\frac{(k+1)C_{k,0}}{(ia-\varepsilon)^{k+1}}I_{a,b}^{\varepsilon,y_0}\Big(\frac{f^{(k)}}{y_\varepsilon^{k+2}}\Big),
\end{align*}
where in the last line we used the property $C_{n,0}=-\frac{1}{2}C_{n-1,0}$ of the coefficients. \medskip

{\it Step 2.} Now we generalize \eqref{Eq_Integral_formula} to every $m\in\mathbb{N}_0$. The induction start $m=0$ was already done in Step 1. For the induction step $m\to m+1$, we assume that \eqref{Eq_Integral_formula} holds for $n\in\mathbb{N}$ and $m\in\mathbb{N}_0$. The computations will now be split into two parts. In {\it Step 2.1} we verify for every $K\in\{0,\dots,n-1\}$ the formula
\begin{equation} \label{Eq_Integral_formula_1}
\begin{split}
I_{a,b}^{\varepsilon,y_0}(f)&=\sum\limits_{k=0}^{n-1}\sum\limits_{l=0}^m\frac{C_{k,l}e^{iab^2-\varepsilon(b-y_0)^2}f^{(k)}(b)}{(ia-\varepsilon)^{k+1+l}b_\varepsilon^{k+1+2l}}+\sum\limits_{k=0}^{K-1}\frac{C_{k,m+1}e^{iab^2-\varepsilon(b-y_0)^2}f^{(k)}(b)}{(ia-\varepsilon)^{k+2+m}b_\varepsilon^{k+3+2m}}  \\
&\quad+\sum\limits_{l=0}^m\frac{C_{n-1,l}}{(ia-\varepsilon)^{n+l}}I_{a,b}^{\varepsilon,y_0}\Big(\frac{f^{(n)}}{y_\varepsilon^{n+2l}}\Big)-\sum\limits_{k=0}^{K-1}\frac{(k+3+2m)C_{k,m+1}}{(ia-\varepsilon)^{k+2+m}}I_{a,b}^{\varepsilon,y_0}\Big(\frac{f^{(k)}}{y_\varepsilon^{k+4+2m}}\Big)  \\
&\quad-\frac{2C_{K,m+1}}{(ia-\varepsilon)^{K+1+m}}I_{a,b}^{\varepsilon,y_0}\Big(\frac{f^{(K)}}{y_\varepsilon^{K+2+2m}}\Big)-\sum\limits_{k=K+1}^{n-1}\frac{(k+1+2m)C_{k,m}}{(ia-\varepsilon)^{k+1+m}}I_{a,b}^{\varepsilon,y_0}\Big(\frac{f^{(k)}}{y_\varepsilon^{k+2+2m}}\Big).
\end{split}
\end{equation}
For $K=0$, this formula reads as
\begin{align*}
I_{a,b}^{\varepsilon,y_0}(f)&=\sum\limits_{k=0}^{n-1}\sum\limits_{l=0}^m\frac{C_{k,l}e^{iab^2-\varepsilon(b-y_0)^2}f^{(k)}(b)}{(ia-\varepsilon)^{k+1+l}b_\varepsilon^{k+1+2l}}+\sum\limits_{l=0}^m\frac{C_{n-1,l}}{(ia-\varepsilon)^{n+l}}I_{a,b}^{\varepsilon,y_0}\Big(\frac{f^{(n)}}{y_\varepsilon^{n+2l}}\Big) \\
&\quad-\frac{2C_{0,m+1}}{(ia-\varepsilon)^{m+1}}I_{a,b}^{\varepsilon,y_0}\Big(\frac{f}{y_\varepsilon^{2+2m}}\Big)-\sum\limits_{k=1}^{n-1}\frac{(k+1+2m)C_{k,m}}{(ia-\varepsilon)^{k+1+m}}I_{a,b}^{\varepsilon,y_0}\Big(\frac{f^{(k)}}{y_\varepsilon^{k+2+2m}}\Big),
\end{align*}
and holds, since with $C_{0,m+1}=\frac{1+2m}{2}C_{0,m}$ it is exactly the induction assumption \eqref{Eq_Integral_formula}. For the step $K\to K+1$, we use the formula \eqref{Eq_Integration_by_parts} with $f^{(K)}$ and $\kappa=K+2+2m$ in the induction assumption \eqref{Eq_Integral_formula_1}, and get
\begin{align*}
I_{a,b}^{\varepsilon,y_0}(f)&=\sum\limits_{k=0}^{n-1}\sum\limits_{l=0}^m\frac{C_{k,l}e^{iab^2-\varepsilon(y-y_0)^2}f^{(k)}(b)}{(ia-\varepsilon)^{k+1+l}b_\varepsilon^{k+1+2l}}+\sum\limits_{k=0}^{K-1}\frac{C_{k,m+1}e^{iab^2-\varepsilon(y-y_0)^2}f^{(k)}(b)}{(ia-\varepsilon)^{k+2+m}b_\varepsilon^{k+3+2m}} \\
&\quad+\sum\limits_{l=0}^m\frac{C_{n-1,l}}{(ia-\varepsilon)^{n+l}}I_{a,b}^{\varepsilon,y_0}\Big(\frac{f^{(n)}}{y_\varepsilon^{n+2l}}\Big)-\sum\limits_{k=0}^{K-1}\frac{(k+3+2m)C_{k,m+1}}{(ia-\varepsilon)^{k+2+m}}I_{a,b}^{\varepsilon,y_0}\Big(\frac{f^{(k)}}{y_\varepsilon^{k+4+2m}}\Big) \\
&\quad+\frac{C_{K,m+1}}{(ia-\varepsilon)^{K+2+m}}\bigg(\frac{e^{iab^2-\varepsilon(y-y_0)^2}f^{(K)}(b)}{b_\varepsilon^{K+3+2m}}-(K+3+2m)I_{a,b}^{\varepsilon,y_0}\Big(\frac{f^{(K)}}{y_\varepsilon^{K+4+2m}}\Big) \\
&\hspace{3.5cm}+I_{a,b}^{\varepsilon,y_0}\Big(\frac{f^{(K+1)}}{y_\varepsilon^{K+3+2m}}\Big)\bigg)-\sum\limits_{k=K+1}^{n-1}\frac{(k+1+2m)C_{k,m}}{(ia-\varepsilon)^{k+1+m}}I_{a,b}^{\varepsilon,y_0}\Big(\frac{f^{(k)}}{y_\varepsilon^{k+2+2m}}\Big).
\end{align*}
The relation $C_{K,m+1}-(K+2+2m)C_{K+1,m}=-2C_{K+1,m+1}$ of the coefficients \eqref{Eq_Coefficients} simplifies this equation to the stated formula \eqref{Eq_Integral_formula_1}, namely
\begin{align*}
I_{a,b}^{\varepsilon,y_0}(f)&=\sum\limits_{k=0}^{n-1}\sum\limits_{l=0}^m\frac{C_{k,l}e^{iab^2-\varepsilon(y-y_0)^2}f^{(k)}(b)}{(ia-\varepsilon)^{k+1+l}b_\varepsilon^{k+1+2l}}+\sum\limits_{k=0}^K\frac{C_{k,m+1}e^{iab^2-\varepsilon(y-y_0)^2}f^{(k)}(b)}{(ia-\varepsilon)^{k+2+m}b_\varepsilon^{k+3+2m}} \\
&\quad+\sum\limits_{l=0}^m\frac{C_{n-1,l}}{(ia-\varepsilon)^{n+l}}I_{a,b}^{\varepsilon,y_0}\Big(\frac{f^{(n)}}{y_\varepsilon^{n+2l}}\Big)-\sum\limits_{k=0}^K\frac{(k+3+2m)C_{k,m+1}}{(ia-\varepsilon)^{k+2+m}}I_{a,b}^{\varepsilon,y_0}\Big(\frac{f^{(k)}}{y_\varepsilon^{k+4+2m}}\Big) \\
&\quad-\frac{2C_{K+1,m+1}}{(ia-\varepsilon)^{K+2+m}}I_{a,b}^{\varepsilon,y_0}\Big(\frac{f^{(K+1)}}{y_\varepsilon^{K+3+2m}}\Big)-\sum\limits_{k=K+2}^{n-1}\frac{(k+1+2m)C_{k,m}}{(ia-\varepsilon)^{k+1+m}}I_{a,b}^{\varepsilon,y_0}\Big(\frac{f^{(k)}}{y_\varepsilon^{k+2+2m}}\Big).
\end{align*}
In {\it Step 2.2} we prove the second part of the induction step $m\to m+1$. Choosing the special value $K=n-1$ in \eqref{Eq_Integral_formula} gives
\begin{align*}
I_{a,b}^{\varepsilon,y_0}(f)&=\sum\limits_{k=0}^{n-1}\sum\limits_{l=0}^m\frac{C_{k,l}e^{iab^2-\varepsilon(y-y_0)^2}f^{(k)}(b)}{(ia-\varepsilon)^{k+1+l}b_\varepsilon^{k+1+2l}}+\sum\limits_{k=0}^{n-2}\frac{C_{k,m+1}e^{iab^2-\varepsilon(y-y_0)^2}f^{(k)}(b)}{(ia-\varepsilon)^{k+2+m}b_\varepsilon^{k+3+2m}} \\
&\quad+\sum\limits_{l=0}^m\frac{C_{n-1,l}}{(ia-\varepsilon)^{n+l}}I_{a,b}^{\varepsilon,y_0}\Big(\frac{f^{(n)}}{y_\varepsilon^{n+2l}}\Big)-\sum\limits_{k=0}^{n-2}\frac{(k+3+2m)C_{k,m+1}}{(ia-\varepsilon)^{k+2+m}}I_{a,b}^{\varepsilon,y_0}\Big(\frac{f^{(k)}}{y_\varepsilon^{k+4+2m}}\Big) \\
&\quad-\frac{2C_{n-1,m+1}}{(ia-\varepsilon)^{n+m}}I_{a,b}^{\varepsilon,y_0}\Big(\frac{f^{(n-1)}}{y_\varepsilon^{n+1+2m}}\Big).
\end{align*}
Using a last time the identity \eqref{Eq_Integration_by_parts} with $f^{(n-1)}$ and $\kappa=n+1+2m$ turns this formula into \eqref{Eq_Integral_formula} with $m+1$, namely
\begin{align*}
I_{a,b}^{\varepsilon,y_0}(f)&=\sum\limits_{k=0}^{n-1}\sum\limits_{l=0}^m\frac{C_{k,l}e^{iab^2-\varepsilon(y-y_0)^2}f^{(k)}(b)}{(ia-\varepsilon)^{k+1+l}b_\varepsilon^{k+1+2l}}+\sum\limits_{k=0}^{n-2}\frac{C_{k,m+1}e^{iab^2-\varepsilon(y-y_0)^2}f^{(k)}(b)}{(ia-\varepsilon)^{k+2+m}b_\varepsilon^{k+3+2m}} \\
&\quad+\sum\limits_{l=0}^m\frac{C_{n-1,l}}{(ia-\varepsilon)^{n+l}}I_{a,b}^{\varepsilon,y_0}\Big(\frac{f^{(n)}}{y_\varepsilon^{n+2l}}\Big)-\sum\limits_{k=0}^{n-2}\frac{(k+3+2m)C_{k,m+1}}{(ia-\varepsilon)^{k+2+m}}I_{a,b}^{\varepsilon,y_0}\Big(\frac{f^{(k)}}{y_\varepsilon^{k+4+2m}}\Big) \\
&\quad+\frac{C_{n-1,m+1}}{(ia-\varepsilon)^{n+1+m}}\bigg(\frac{e^{iab^2-\varepsilon(y-y_0)^2}f^{(n-1)}(b)}{b_\varepsilon^{n+2+2m}}-(n+2+2m)I_{a,b}^{\varepsilon,y_0}\Big(\frac{f^{(n-1)}}{y_\varepsilon^{n+3+2m}}\Big) \\
&\hspace{9.5cm}+I_{a,b}^{\varepsilon,y_0}\Big(\frac{f^{(n)}}{y_\varepsilon^{n+2+2m}}\Big)\bigg) \\
&=\sum\limits_{k=0}^{n-1}\sum\limits_{l=0}^m\frac{C_{k,l}e^{iab^2-\varepsilon(y-y_0)^2}f^{(k)}(b)}{(ia-\varepsilon)^{k+1+l}b_\varepsilon^{k+1+2l}}+\sum\limits_{k=0}^{n-1}\frac{C_{k,m+1}e^{iab^2-\varepsilon(y-y_0)^2}f^{(k)}(b)}{(ia-\varepsilon)^{k+2+m}b_\varepsilon^{k+3+2m}} \\
&\quad+\sum\limits_{l=0}^{m+1}\frac{C_{n-1,l}}{(ia-\varepsilon)^{n+l}}I_{a,b}^{\varepsilon,y_0}\Big(\frac{f^{(n)}}{y_\varepsilon^{n+2l}}\Big)-\sum\limits_{k=0}^{n-1}\frac{(k+3+2m)C_{k,m+1}}{(ia-\varepsilon)^{k+2+m}}I_{a,b}^{\varepsilon,y_0}\Big(\frac{f^{(k)}}{y_\varepsilon^{k+4+2m}}\Big). \qedhere
\end{align*}
\end{proof}

In the following theorem we will make use of the formula \eqref{Eq_Integral_formula} to make sense of the integral $\mathcal{I}_{a,b}(f)$ in \eqref{Eq_Iab_formal} for functions $f$ in the space $C_\alpha^n([b,\infty))$.

\begin{thm}\label{thm_Integral}
Let $a\in\mathbb{R}\setminus\{0\}$, $b>0$, $n\in\mathbb{N}$, $\alpha\geq 0$, with $n>\alpha+1$. Then for every $f\in C_\alpha^n([b,\infty))$ the oscillatory integral
\begin{equation}\label{Eq_Integral}
\mathcal{I}_{a,b}(f):=\lim\limits_{\varepsilon\to 0^+}I_{a,b}^{\varepsilon,y_0}(f)=\lim\limits_{\varepsilon\to 0^+}\int_b^\infty e^{-\varepsilon(y-y_0)^2}e^{iay^2}f(y)dy
\end{equation}
exists and is independent of $y_0\in\mathbb{R}$. Moreover, for every $0<\varepsilon\leq 1$ one has
\begin{equation}\label{Eq_D}
|I_{a,b}^{\varepsilon,y_0}(f)|\leq D_{a,b,n,\alpha}\Vert f\Vert_{C_\alpha^n([b,\infty))}\quad\text{and}\quad|\mathcal{I}_{a,b}(f)|\leq D_{a,b,n,\alpha}\Vert f\Vert_{C_\alpha^n([b,\infty))},
\end{equation}
where $D_{a,b,n,\alpha}\geq 0$  depends continuously on $a,b$ and admits the asymptotics
\begin{equation}\label{Eq_D_asymptotics}
D_{a,b,n,\alpha}=\mathcal{O}\Big(\frac{1}{|a|}\Big),\quad\text{as }|a|\to\infty.
\end{equation}
\end{thm}

\begin{proof}
Let us fix $y_0\in\mathbb{R}$ and $0<\varepsilon\leq 1$. According to \eqref{Eq_Integral_formula}, with the choice $m=n-1$, we have the representation
\begin{equation}\label{Eq_Integral_1}
\begin{split}
I_{a,b}^{\varepsilon,y_0}(f)&=\sum\limits_{k=0}^{n-1}\sum\limits_{l=0}^{n-1}\frac{C_{k,l}e^{iab^2-\varepsilon(y-y_0)^2}f^{(k)}(b)}{(ia-\varepsilon)^{k+1+l}b_\varepsilon^{k+1+2l}}+\sum\limits_{l=0}^{n-1}\frac{C_{n-1,l}}{(ia-\varepsilon)^{n+l}}I_{a,b}^{\varepsilon,y_0}\Big(\frac{f^{(n)}}{y_\varepsilon^{n+2l}}\Big) \\
&\quad-\sum\limits_{k=0}^{n-1}\frac{(k+2n-1)C_{k,n-1}}{(ia-\varepsilon)^{k+n}}I_{a,b}^{\varepsilon,y_0}\Big(\frac{f^{(k)}}{y_\varepsilon^{k+2n}}\Big).
\end{split}
\end{equation}
To show that the limit $\varepsilon\to 0^+$ of the right hand side exists it suffices to verify the existence of the limits
\begin{equation}\label{limjussi1}
\lim\limits_{\varepsilon\to 0^+}I_{a,b}^{\varepsilon,y_0}\Big(\frac{f^{(n)}}{y_\varepsilon^{n+2l}}\Big),\qquad\text{and}\qquad\lim\limits_{\varepsilon\to 0^+}I_{a,b}^{\varepsilon,y_0}\Big(\frac{f^{(k)}}{y_\varepsilon^{k+2n}}\Big), \qquad k,l\in\{0,\dots,n-1\}.
\end{equation} 
For the first limit in \eqref{limjussi1} we use the upper bound \eqref{Eq_fk_estimate} and obtain for every $l\in\{0,\dots,n-1\}$ the estimate
\begin{equation}\label{Eq_Integral_5}
\Big|\frac{f^{(n)}(y)}{y_\varepsilon^{n+2l}}\Big|\leq\frac{\Vert f\Vert_{C_\alpha^n([b,\infty))}y^\alpha}{|y_\varepsilon|^{n+2l}}\leq\frac{|ia-\varepsilon|^{n+2l}\Vert f\Vert_{C_\alpha^n([b,\infty))}}{|a|^{n+2l}y^{n+2l-\alpha}}\leq\frac{|ia-\varepsilon|^{n+2l}\Vert f\Vert_{C_\alpha^n([b,\infty))}}{|a|^{n+2l}b^{2l}y^{n-\alpha}},
\end{equation}
where in the second inequality we used
\begin{equation}\label{Eq_Integral_9}
|y_\varepsilon|=\Big|y+\frac{\varepsilon y_0}{ia-\varepsilon}\Big|=\frac{|iay-\varepsilon(y-y_0)|}{|ia-\varepsilon|}\geq\frac{|a|y}{|ia-\varepsilon|}.
\end{equation}
Since $n-\alpha>1$ by assumption, the estimate \eqref{Eq_Integral_5} shows that $\frac{f^{(n)}(y)}{y_\varepsilon^{n+2l}}$ is integrable on $[b,\infty)$. Moreover, for every $0<\varepsilon\leq 1$ we can estimate $|ia-\varepsilon|\leq \vert a\vert +1$ and make the upper bound \eqref{Eq_Integral_5} independent of $\varepsilon$. Then the dominated convergence theorem implies the existence of the limit
\begin{equation*}
\lim\limits_{\varepsilon\to 0^+}I_{a,b}^{\varepsilon,y_0}\Big(\frac{f^{(n)}}{y_\varepsilon^{n+2l}}\Big)=\lim\limits_{\varepsilon\to 0^+}\int_b^\infty e^{iay^2-\varepsilon(y-y_0)^2}\frac{f^{(n)}(y)}{y_\varepsilon^{n+2l}}dy=\int_b^\infty e^{iay^2}\frac{f^{(n)}(y)}{y^{n+2l}}dy.
\end{equation*}
Moreover, from \eqref{Eq_Integral_5} we obtain the upper bound
\begin{equation}\label{Eq_Integral_8}
\Big|I_{a,b}^{\varepsilon,y_0}\Big(\frac{f^{(n)}}{y_\varepsilon^{n+2l}}\Big)\Big|\leq\frac{|ia-\varepsilon|^{n+2l}\Vert f\Vert_{C_\alpha^n([b,\infty))}}{|a|^{n+2l}b^{2l}}\int_b^\infty\frac{1}{y^{n-\alpha}}dy=\frac{|ia-\varepsilon|^{n+2l}\Vert f\Vert_{C_\alpha^n([b,\infty))}}{(n-\alpha-1)|a|^{n+2l}b^{2l+n-\alpha-1}}.
\end{equation}
To show the existence of the second limit in \eqref{limjussi1} we estimate in the same way as in \eqref{Eq_Integral_5} for every $k\in\{0,\dots,n-1\}$ 
\begin{equation}\label{Eq_Integral_4}
\Big|\frac{f^{(k)}(y)}{y_\varepsilon^{k+2n}}\Big|\leq\frac{\Vert f\Vert_{C_\alpha^n([b,\infty))}y^{n-k+\alpha}}{|y_\varepsilon|^{k+2n}}\leq\frac{|ia-\varepsilon|^{k+2n}\Vert f\Vert_{C_\alpha^n([b,\infty))}}{a^{k+2n}y^{2k+n-\alpha}}\leq\frac{|ia-\varepsilon|^{k+2n}\Vert f\Vert_{C_\alpha^n([b,\infty))}}{a^{k+2n}b^{2k}y^{n-\alpha}}.
\end{equation}
Since $n-\alpha>1$ by assumption, this estimate shows that $\frac{f^{(k)}}{y_\varepsilon^{k+2n}}$ is integrable on $[b,\infty)$, and the dominated convergence theorem gives the existence of the limit
\begin{equation*}
\lim\limits_{\varepsilon\to 0^+}I_{a,b}^{\varepsilon,y_0}\Big(\frac{f^{(k)}}{y_\varepsilon^{k+2n}}\Big)=\lim\limits_{\varepsilon\to 0^+}\int_b^\infty e^{iay^2-\varepsilon(y-y_0)^2}\frac{f^{(k)}(y)}{y_\varepsilon^{k+2n}}dy=\int_b^\infty e^{iay^2}\frac{f^{(k)}(y)}{y^{k+2n}}dy.
\end{equation*}
Moreover, from \eqref{Eq_Integral_4} we obtain the upper bound 
\begin{equation}\label{Eq_Integral_7}
\Big|I_{a,b}^{\varepsilon,y_0}\Big(\frac{f^{(k)}}{y_\varepsilon^{k+2n}}\Big)\Big|\leq\frac{|ia-\varepsilon|^{k+2n}\Vert f\Vert_{C_\alpha^n([b,\infty))}}{|a|^{k+2n}b^{2k}}\int_b^\infty\frac{1}{y^{n-\alpha}}dy=\frac{|ia-\varepsilon|^{k+2n}\Vert f\Vert_{C_\alpha^n([b,\infty))}}{(n-\alpha-1)|a|^{k+2n}b^{2k+n-\alpha-1}}.
\end{equation}
Now \eqref{limjussi1} implies that also the limit of the sum \eqref{Eq_Integral_1} exists and is given by
\begin{equation}\label{Eq_Integral_2}
\begin{split}
\mathcal{I}_{a,b}(f):=\lim\limits_{\varepsilon\to 0^+}I_{a,b}^{\varepsilon,y_0}(f)&=\sum\limits_{k=0}^{n-1}\sum\limits_{l=0}^{n-1}\frac{C_{k,l}e^{iab^2}f^{(k)}(b)}{(ia)^{k+1+l}b^{k+1+2l}}+\sum\limits_{l=0}^{n-1}\frac{C_{n-1,l}}{(ia)^{n+l}}\int_b^\infty e^{iay^2}\frac{f^{(n)}(y)}{y^{n+2l}}dy  \\
&\quad-\sum\limits_{k=0}^{n-1}\frac{(k+2n-1)C_{k,n-1}}{(ia)^{k+n}}\int_b^\infty e^{iay^2}\frac{f^{(k)}(y)}{y^{k+2n}}dy. 
\end{split}
\end{equation}
With the estimates \eqref{Eq_Integral_9}, \eqref{Eq_Integral_8}, \eqref{Eq_Integral_7}, with $|f^{(k)}(b)|\leq\Vert f\Vert_{C_\alpha^n([b,\infty))}b^{n-k+\alpha}$ by \eqref{Eq_fk_estimate}, and with $|b_\varepsilon|\geq\frac{|a|b}{|ia-\varepsilon|}$ as in \eqref{Eq_Integral_9}, we can also estimate \eqref{Eq_Integral_1} by
\begin{align*}
|I_{a,b}^{\varepsilon,y_0}(f)|&\leq\sum\limits_{k=0}^{n-1}\sum\limits_{l=0}^{n-1}\frac{|C_{k,l}||ia-\varepsilon|^l\Vert f\Vert_{C_\alpha^n([b,\infty))}}{|a|^{k+1+2l}b^{2k+2l-n-\alpha+1}}+\sum\limits_{l=0}^{n-1}\frac{|C_{n-1,l}||ia-\varepsilon|^l\Vert f\Vert_{C_\alpha^n([b,\infty))}}{(n-\alpha-1)|a|^{n+2l}b^{2l+n-\alpha-1}} \\
&\quad+\sum\limits_{k=0}^{n-1}\frac{(k+2n-1)|C_{k,n-1}||ia-\varepsilon|^n\Vert f\Vert_{C_\alpha^n([b,\infty))}}{(n-\alpha-1)|a|^{k+2n}b^{2k+n-\alpha-1}} \\
&\leq\bigg(\sum\limits_{k=0}^{n-1}\sum\limits_{l=0}^{n-1}\frac{|C_{k,l}|(|a|+1)^l}{|a|^{k+1+2l}b^{2k+2l-n-\alpha+1}}+\sum\limits_{l=0}^{n-1}\frac{|C_{n-1,l}|(|a|+1)^l}{(n-\alpha-1)|a|^{n+2l}b^{2l+n-\alpha-1}} \\
&\qquad+\sum\limits_{k=0}^{n-1}\frac{(k+2n-1)|C_{k,n-1}|(|a|+1)^n}{(n-\alpha-1)|a|^{k+2n}b^{2k+n-\alpha-1}}\bigg)\Vert f\Vert_{C_\alpha^n([b,\infty))},
\end{align*}
where in the second inequality we used $\varepsilon\leq 1$ to estimate $|ia-\varepsilon|\leq|a|+1$. This estimate shows \eqref{Eq_D} with $D_{a,b,n,\alpha}$ being the term inside the parentheses. The continuous dependency on $a$ and $b$ of this coefficient is obvious. Since $n>\alpha+1\geq 1$, one also obtains the convergence
\begin{equation*}
\lim\limits_{|a|\to\infty}\big(|a|D_{a,b,n,\alpha}\big)=\frac{|C_{0,0}|}{b^{-n-\alpha+1}}=\frac{b^{n+\alpha-1}}{2},
\end{equation*}
and hence we conclude the asymptotics $D_{a,b,n,\alpha}=\mathcal{O}(\frac{1}{|a|})$ as $|a|\to\infty$.
\end{proof}

Next, we will illustrate for the function $f(y)=e^{i\kappa y}$, $\kappa\in\mathbb{R}$, how the integral $\mathcal{I}_{a,b}(e^{i\kappa y})$ in \eqref{Eq_Integral} can be computed. Here we shall make use of the error function 
\begin{equation}\label{erfi}
\erf(z)=\frac{2}{\sqrt{\pi}}\int_0^ze^{-\xi^2}d\xi,\qquad z\in\mathbb{C},
\end{equation}
and we agree to cut the complex square root along the negative real axis, so that $\Re(\sqrt{z})>0$ for every $z\in\mathbb{C}\setminus(-\infty,0]$.

\begin{exam}
For $\kappa\in\mathbb{R}$ consider the exponential function $f(y):=e^{i\kappa y}$. The derivatives are $f^{(n)}(y)=(i\kappa)^ne^{i\kappa y}$ and hence $f\in C_0^n([b,\infty))$ for every $n\in\mathbb{N}_0$. Choosing in particular $n\geq 2$, the oscillatory integral \eqref{Eq_Integral} is for every $a\in\mathbb{R}\setminus\{0\}$, $b>0$ well defined, and has the explicit value
\begin{align*}
\mathcal{I}_{a,b}(e^{i\kappa y})&=\lim\limits_{\varepsilon\to 0^+}\int_b^\infty e^{-\varepsilon y^2}e^{iay^2+i\kappa y}dy \\
&=\lim\limits_{\varepsilon\to 0^+}\frac{\sqrt{\pi}}{2\sqrt{\varepsilon-ia}}e^{-\frac{\kappa^2}{4(\varepsilon-ia)}}\erf\Big(y\sqrt{\varepsilon-ia}-\frac{i\kappa}{2\sqrt{\varepsilon-ia}}\Big)\bigg|_{y=b}^\infty.
\end{align*}
Since $\Arg(\sqrt{\varepsilon-ia})\in(-\frac{\pi}{4},\frac{\pi}{4})$, the above error function in the limit $y\to\infty$ is equal to $1$, see, e.g., \cite[Eq.7.1.16]{AS72}. Hence we obtain for the oscillatory integral
\begin{align*}
\mathcal{I}_{a,b}(e^{i\kappa y})&=\lim\limits_{\varepsilon\to 0^+}\frac{\sqrt{\pi}}{2\sqrt{\varepsilon-ia}}e^{-\frac{\kappa^2}{4(\varepsilon-ia)}}\bigg(1-\erf\Big(b\sqrt{\varepsilon-ia}-\frac{i\kappa}{2\sqrt{\varepsilon-ia}}\Big)\bigg) \\
&=\frac{\sqrt{\pi}}{2\sqrt{-ia}}e^{\frac{\kappa^2}{4ia}}\bigg(1-\erf\Big(b\sqrt{-ia}-\frac{i\kappa}{2\sqrt{-ia}}\Big)\bigg).
\end{align*}
We note that $\mathcal{I}_{a,b}(e^{i\kappa y})$ is the natural generalization of the absolutely convergent Gauss integral
\begin{equation*}
\int_b^\infty e^{-\alpha y^2+i\kappa y}dy=\frac{\sqrt{\pi}}{2\sqrt{\alpha}}e^{-\frac{\kappa^2}{4\alpha}}\bigg(1-\erf\Big(b\sqrt{\alpha}-\frac{i\kappa}{2\sqrt{\alpha}}\Big)\bigg),\qquad\Re(\alpha)>0,
\end{equation*}
to purely imaginary values $\alpha=-ia\in i\mathbb{R}\setminus\{0\}$.
\end{exam}

\begin{cor}
Let $a\in\mathbb{R}\setminus\{0\}$, $b>0$, $n\in\mathbb{N}$, $\alpha\geq 0$, with $n>\alpha+1$, and assume that $f\in C_\alpha^n([b,\infty))$ satisfies for every $k\in\{0,\dots,n-1\}$ the asymptotic condition
\begin{equation}\label{Eq_Riemann_assumption}
f^{(k)}(y)=o(y^{k+1})\quad\text{as }y\to\infty.
\end{equation} 
Then the oscillatory integral $\mathcal{I}_{a,b}(f)$ in \eqref{Eq_Integral} can be represented as the improper Riemann integral
\begin{equation*}
\mathcal{I}_{a,b}(f)=\lim\limits_{R\to\infty}\int_b^Re^{iay^2}f(y)dy.
\end{equation*}
\end{cor}

\begin{proof}
The formula \eqref{Eq_Integral_formula} with $m=n-1$ and $y_0=0$, i.e. $y_\varepsilon=y$ and $b_\varepsilon=b$, is given by
\begin{equation}\label{Eq_Integral_Riemann_2}
\begin{split}
I_{a,b}^{\varepsilon,0}(f)&=\sum\limits_{k=0}^{n-1}\sum\limits_{l=0}^{m-1}\frac{C_{k,l}e^{(ia-\varepsilon)b^2}f^{(k)}(b)}{(ia-\varepsilon)^{k+1+l}b^{k+1+2l}}+\sum\limits_{l=0}^{n-1}\frac{C_{n-1,l}}{(ia-\varepsilon)^{n+l}}I_{a,b}^{\varepsilon,0}\Big(\frac{f^{(n)}}{y^{n+2l}}\Big) \\
&\quad-\sum\limits_{k=0}^{n-1}\frac{(k-1+2n)C_{k,n-1}}{(ia-\varepsilon)^{k+n}}I_{a,b}^{\varepsilon,0}\Big(\frac{f^{(k)}}{y^{k+2n}}\Big).
\end{split}
\end{equation}
Let now $R>b$ be arbitrary, and subtract the identity \eqref{Eq_Integral_Riemann_2} from the same one with $b$ replaced by $R$. This gives
\begin{align*}
\int_b^Re^{(ia-\varepsilon)y^2}f(y)dy&=I_{a,b}^{\varepsilon,0}(f)-I_{a,R}^{\varepsilon,0}(f) \\
&=\sum\limits_{k=0}^{n-1}\sum\limits_{l=0}^{n-1}\frac{C_{k,l}}{(ia-\varepsilon)^{k+1+l}}\Big(\frac{e^{(ia-\varepsilon)b^2}f^{(k)}(b)}{b^{k+1+2l}}-\frac{e^{(ia-\varepsilon)R^2}f^{(k)}(R)}{R^{k+1+2l}}\Big) \\
&\quad+\sum\limits_{l=0}^{n-1}\frac{C_{n-1,l}}{(ia-\varepsilon)^{n+l}}\int_b^Re^{(ia-\varepsilon)y^2}\frac{f^{(n)}(y)}{y^{n+2l}}dy \\
&\quad-\sum\limits_{k=0}^{n-1}\frac{(k-1+2n)C_{k,n-1}}{(ia-\varepsilon)^{k+n}}\int_b^Re^{(ia-\varepsilon)y^2}\frac{f^{(k)}(y)}{y^{k+2n}}dy.
\end{align*}
Since the integration intervals are all finite, each integrand is absolutely integrable also for $\varepsilon=0$. Thus, the limit $\varepsilon\to 0^+$ exists and has the form
\begin{equation}\label{Eq_Integral_Riemann_1}
\begin{split}
\int_b^Re^{iay^2}f(y)dy&=\sum\limits_{k=0}^{n-1}\sum\limits_{l=0}^{n-1}\frac{C_{k,l}}{(ia)^{k+1+l}}\Big(\frac{e^{iab^2}f^{(k)}(b)}{b^{k+1+2l}}-\frac{e^{iaR^2}f^{(k)}(R)}{R^{k+1+2l}}\Big) \\
&\quad+\sum\limits_{l=0}^{n-1}\frac{C_{n-1,l}}{(ia)^{n+l}}\int_b^Re^{iay^2}\frac{f^{(n)}(y)}{y^{n+2l}}dy\\
&\quad -\sum\limits_{k=0}^{n-1}\frac{(k+2n-1)C_{k,n-1}}{(ia)^{k+n}}\int_b^Re^{iay^2}\frac{f^{(k)}(y)}{y^{k+2n}}dy. 
\end{split}
\end{equation}
From \eqref{Eq_fk_estimate} and $n>\alpha+1$ it follows that $\frac{f^{(n)}}{y^{n+2l}},\frac{f^{(k)}}{y^{k+2n}}\in L^1([b,\infty))$, for every $l,k\in\{0,\dots,n-1\}$, and hence we have
\begin{equation*}
\lim\limits_{R\to\infty}\int_b^Re^{iay^2}\frac{f^{(n)}(y)}{y^{n+2l}}dy=\int_b^\infty e^{iay^2}\frac{f^{(n)}(y)}{y^{n+2l}}dy
\end{equation*}
and
\begin{equation*}
\lim\limits_{R\to\infty}\int_b^Re^{iay^2}\frac{f^{(k)}(y)}{y^{k+2n}}dy=\int_b^\infty e^{iay^2}\frac{f^{(k)}(y)}{y^{k+2n}}dy.
\end{equation*}
From the asymptotic conditions \eqref{Eq_Riemann_assumption} we obtain
\begin{equation*}
\lim\limits_{R\to\infty}\frac{f^{(k)}(R)}{R^{k+1+2l}}=0,\qquad k,l\in\{0,\dots,n-1\},
\end{equation*}
and therefore \eqref{Eq_Integral_Riemann_1} in the limit $R\to\infty$ becomes
\begin{align*}
\lim\limits_{R\to\infty}\int_b^Re^{iay^2}f(y)dy&=\sum\limits_{k=0}^{n-1}\sum\limits_{l=0}^{n-1}\frac{C_{k,l}e^{iab^2}f^{(k)}(b)}{(ia)^{k+1+l}b^{k+1+2l}}+\sum\limits_{l=0}^{n-1}\frac{C_{n-1,l}}{(ia)^{n+l}}\int_b^\infty e^{iay^2}\frac{f^{(n)}(y)}{y^{n+2l}}dy \\
&\quad-\sum\limits_{k=0}^{n-1}\frac{(k+2n-1)C_{k,n-1}}{(ia)^{k+n}}\int_b^\infty e^{iay^2}\frac{f^{(k)}(y)}{y^{k+2n}}dy.
\end{align*}
Finally, since the above expression coincides with the representation \eqref{Eq_Integral_2} of the oscillatory integral $\mathcal{I}_{a,b}(f)$ we conclude
\begin{equation*}
\lim\limits_{R\to\infty}\int_b^Re^{iay^2}f(y)dy=\mathcal{I}_{a,b}(f)=\lim\limits_{\varepsilon\to 0^+}I_{a,b}^{\varepsilon,0}(f)=\lim\limits_{\varepsilon\to 0^+}\int_b^\infty e^{-\varepsilon y^2}e^{iay^2}f(y)dy. \qedhere
\end{equation*}
\end{proof}

In the next proposition we show that for a convergent sequence of functions $(f_m)_m$ also the sequence of oscillatory integrals $\mathcal{I}_{a,b}(f_m)$ converges.

\begin{prop}\label{contiprop}
Let $a\in\mathbb{R}\setminus\{0\}$, $b>0$, $n\in\mathbb{N}$, $\alpha\geq 0$ with $n>\alpha+1$, and assume that $f,(f_m)_m\in C_\alpha^n([b,\infty))$ are such that
\begin{equation*}
\lim\limits_{m\to\infty}\Vert f-f_m\Vert_{C^n_\alpha([b,\infty))}=0.
\end{equation*}
Then also the oscillatory integrals $\mathcal{I}_{a,b}(f_m)$ converge and one has 
\begin{equation*}
\mathcal{I}_{a,b}(f)=\mathcal{I}_{a,b}\Big(\lim_{m\to\infty}f_m\Big)=\lim\limits_{m\to\infty}\mathcal{I}_{a,b}(f_m).
\end{equation*}
\end{prop}

\begin{proof}
For every $m\in\mathbb{N}$ and $0<\varepsilon\leq 1$ we obtain from \eqref{Eq_D} the $\varepsilon$-uniform convergence
\begin{equation*}
|I_{a,b}^{\varepsilon,y_0}(f_m)-I_{a,b}^{\varepsilon,y_0}(f)|=|I_{a,b}^{\varepsilon,y_0}(f_m-f)|\leq D_{a,b,n,\alpha}\Vert f-f_m\Vert_{C_\alpha^n([b,\infty))}\to 0,\qquad\text{as }m\to\infty.
\end{equation*}
Hence we are allowed to interchange the limits $m\to\infty$ and $\varepsilon\to 0^+$, and get
\begin{equation*}
\lim\limits_{m\to\infty}\mathcal{I}_{a,b}(f_m)=\lim\limits_{m\to\infty}\lim\limits_{\varepsilon\to 0^+}I_{a,b}^{\varepsilon,y_0}(f_m)=\lim\limits_{\varepsilon\to 0^+}\lim\limits_{m\to\infty}I_{a,b}^{\varepsilon,y_0}(f_m)=\lim\limits_{\varepsilon\to 0^+}I_{a,b}^{\varepsilon,y_0}(f)=\mathcal{I}_{a,b}(f). \qedhere
\end{equation*}
\end{proof}

The following Theorem~\ref{thm_Integral_derivative} provides conditions under which the oscillatory integral \eqref{Eq_Integral} is absolutely continuous with respect to an additional parameter $s$, and we obtain a formula for its derivative (defined almost everywhere). Recall, that a function $f:I\to\mathbb{C}$, defined on some open interval $I\subseteq\mathbb{R}$, is said to be absolutely continuous if there exists $g\in L_\loc^1(I)$, such that
\begin{equation*}
f(s_2)-f(s_1)=\int_{s_1}^{s_2}g(s)ds,\qquad\text{for every }s_1<s_2\in I.
\end{equation*}
In this case the function $f$ is differentiable almost everywhere with derivative $f'=g$. We shall denote the space of absolutely continuous functions on $I$ by $\AC(I)$.

\begin{thm}\label{thm_Integral_derivative}
Let $I\subseteq\mathbb{R}$ be an open interval and $a,b\in\AC(I)$ with $a(s)\neq 0$ and $b(s)\geq b_0>0$ for every $s\in I$ and some positive constant $b_0$. Assume that $f:I\times[b_0,\infty)\to\mathbb{C}$ satisfies the following conditions (i) and (ii).

\begin{enumerate}
\item[(i)] $f(\,\cdot\,,y)\in\AC(I)$ for every $y\in[b_0,\infty)$;

\item[(ii)] There exist $n\in\mathbb{N}$, $\alpha\geq 0$ with $n>\alpha+3$, such that
\begin{equation*}
f(s,\cdot\,)\in C_\alpha^n([b_0,\infty))\quad\text{and}\quad\frac{\partial f}{\partial s}(s,\cdot\,)\in C_{\alpha+2}^n([b_0,\infty))\quad{\rm f.a.e.}\;s\in I,
\end{equation*}
and the norms $\Vert f(s,\cdot\,)\Vert_{C_\alpha^n([b_0,\infty))}$ and $\Vert\frac{\partial f}{\partial s}(s,\cdot\,)\Vert_{C_{\alpha+2}^n([b_0,\infty))}$ are locally bounded on $I$.
\end{enumerate}

\noindent Then the oscillatory integral
\begin{equation}\label{Eq_Integral_derivative}
\psi(s):=\mathcal{I}_{a(s),b(s)}(f(s,\cdot\,))=\lim\limits_{\varepsilon\to 0^+}I_{a(s),b(s)}^{\varepsilon,y_0}(f(s,\cdot\,))=\lim\limits_{\varepsilon\to 0^+}\int_{b(s)}^\infty e^{-\varepsilon(y-y_0)^2}e^{ia(s)y^2}f(s,y)dy
\end{equation}
exists for every $s\in I$, is independent of $y_0\in\mathbb{R}$, and defines a function $\psi\in\AC(I)$. Furthermore, the derivative of $\psi$ exists for almost every $s\in I$ and is given by
\begin{equation*}
\frac{d}{ds}\psi(s)=-b'(s)e^{ia(s)b(s)^2}f(s,b(s))+\lim\limits_{\varepsilon\to 0^+}\int_{b(s)}^\infty e^{-\varepsilon(y-y_0)^2}\frac{d}{ds}\big(e^{ia(s)y^2}f(s,y)\big)dy.
\end{equation*}
\end{thm}

\begin{proof}
Since the oscillatory integrals are independent of the choice of $y_0\in\mathbb{R}$ (see Theorem~\ref{thm_Integral}) we will choose $y_0=0$ in this proof. For almost every $s\in I$, we can split the $s$-derivative of the integrand in \eqref{Eq_Integral_derivative} into
\begin{equation}\label{Eq_Integral_derivative_2}
\frac{d}{ds}\big(e^{ia(s)y^2}f(s,y)\big)=e^{ia(s)y^2} g(s,y),\qquad\text{where}\;g(s,y):=ia'(s)y^2f(s,y)+\frac{\partial f}{\partial s}(s,y).
\end{equation}
As $f(s,\cdot\,)\in C_\alpha^n([b_0,\infty))$ it follows from Corollary~\ref{cor_ypf} that $y^2f(s,\cdot\,)\in C_{\alpha+2}^n([b_0,\infty))$, with norm bounded by
\begin{equation*}
\Vert y^2f(s,\cdot\,)\Vert_{C_{\alpha+2}^n([b_0,\infty))}\leq(n+3)^n\Vert f(s,\cdot\,)\Vert_{C_\alpha^n([b_0,\infty))}.
\end{equation*}
Together with the assumption $\frac{\partial f}{\partial s}(s,\cdot\,)\in C_{\alpha+2}^n([b_0,\infty))$ we conclude $g(s,\cdot\,)\in C_{\alpha+2}^n([b_0,\infty))$, with norm bounded by
\begin{equation*}
\Vert g(s,\cdot\,)\Vert_{C_{\alpha+2}^n([b_0,\infty))}\leq|a'(s)|(n+3)^n\Vert f(s,\cdot\,)\Vert_{C_\alpha^n([b_0,\infty))}+\Big\Vert\frac{\partial f}{\partial s}(s,\cdot\,)\Big\Vert_{C_{n+2}^\alpha([b_0,\infty))}.
\end{equation*}
Since $\Vert f(s,\cdot\,)\Vert_{C_\alpha^n([b_0,\infty))}$ and $\Vert\frac{\partial f}{\partial s}(s,\cdot\,)\Vert_{C_{n+2}^\alpha([b_0,\infty))}$ are locally bounded by assumption, and since $a'\in L^1_\loc(I)$ due to $a\in\AC(I)$, it follows that the norm $\Vert g(s,\cdot\,)\Vert_{C_{\alpha+2}^n([b_0,\infty))}$ admits a locally integrable upper bound. By Corollary~\ref{cor_Reducing_interval} we have $g(s,\cdot\,)\in C_{\alpha+2}^n([b(s),\infty))$, with locally integrable upper bound
\begin{equation*}
\Vert g(s,\cdot\,)\Vert_{C_{\alpha+2}^n([b(s),\infty))}\leq(n+1)\Vert g(s,\cdot\,)\Vert_{C_{\alpha+2}^n([b_0,\infty))}.
\end{equation*}
Our assumption $n>(\alpha+2)+1$ ensures that Theorem~\ref{thm_Integral} can be applied in the present situation and hence we conclude for almost every $s\in I$ the existence of the oscillatory integral
\begin{equation*}
\mathcal{I}_{a(s),b(s)}(g(s,\cdot\,))=\lim\limits_{\varepsilon\to 0^+}I_{a(s),b(s)}^{\varepsilon,0}(g(s,\cdot\,))=\lim\limits_{\varepsilon\to 0^+}\int_{b(s)}^\infty e^{(ia(s)-\varepsilon)y^2}g(s,y)dy,
\end{equation*}
and for every $0<\varepsilon\leq 1$ we have the estimate
\begin{equation}\label{Eq_Integral_derivative_6}
\big|I_{a(s),b(s)}^{\varepsilon,0}(g(s,\cdot\,))\big|\leq D_{a(s),b(s),n,\alpha+2}(n+1)\Vert g(s,\cdot\,)\Vert_{C_{\alpha+2}^n([b_0,\infty))}.
\end{equation}
It follows from the continuity of $a(s),b(s)$ and the continuous dependency of $D_{a,b,n,\alpha+2}$ on $a,b$ in Theorem~\ref{thm_Integral}, that $s\mapsto D_{a(s),b(s),n,\alpha+2}$ is continuous. In particular, it follows that the right hand side of \eqref{Eq_Integral_derivative_6} is locally integrable over any compact subset $[s_0,s_1]\subseteq I$. The dominated convergence theorem then ensures that
\begin{equation}\label{Eq_Integral_derivative_1}
\int_{s_0}^{s_1}\mathcal{I}_{a(s),b(s)}(g(s,\cdot\,))ds=\int_{s_0}^{s_1}\lim\limits_{\varepsilon\to 0^+}I_{a(s),b(s)}^{\varepsilon,0}(g(s,\cdot\,))ds=\lim\limits_{\varepsilon\to 0^+}\int_{s_0}^{s_1}I_{a(s),b(s)}^{\varepsilon,0}(g(s,\cdot\,))ds.
\end{equation}
Next, the derivative of an integral with parameter dependent integrand and boundary is given by
\begin{align*}
\frac{d}{ds}\int_{b(s)}^\infty e^{(ia(s)-\varepsilon)y^2}f(s,y)dy&=-b'(s)e^{(ia(s)-\varepsilon)b(s)^2}f(s,b(s))+\int_{b(s)}^\infty e^{-\varepsilon y^2}\frac{d}{ds}\big(e^{ia(s)y^2}f(s,y)\big)dy \\
&=-b'(s)e^{(ia(s)-\varepsilon)b(s)^2}f(s,b(s))+I_{a(s),b(s)}^{\varepsilon,0}(g(s,\cdot\,)).
\end{align*}
Hence, equation \eqref{Eq_Integral_derivative_1} can be rewritten as
\begin{align*}
\int_{s_0}^{s_1}\mathcal{I}_{a(s),b(s)}(g(s,\cdot\,))ds&=\lim\limits_{\varepsilon\to 0^+}\int_{s_0}^{s_1}\hspace{-0.1cm}\bigg(\hspace{-0.05cm}b'(s)e^{(ia(s)-\varepsilon)b(s)^2}f(s,b(s))+\frac{d}{ds}\int_{b(s)}^\infty e^{(ia(s)-\varepsilon)y^2}f(s,y)dy\hspace{-0.05cm}\bigg)ds \\
&=\int_{s_0}^{s_1}b'(s)e^{ia(s)b(s)^2}f(s,b(s))ds+\lim\limits_{\varepsilon\to 0^+}\bigg[\int_{b(s)}^\infty e^{(ia(s)-\varepsilon)y^2}f(s,y)dy\bigg]\bigg|_{s=s_0}^{s_1} \\
&=\int_{s_0}^{s_1}b'(s)e^{ia(s)b(s)^2}f(s,b(s))ds+\psi(s)\Big|_{s=s_0}^{s_1},
\end{align*}
where in the last equation we used the definition of $\psi(s)$, and in the second equation we were allowed to carry the limit $\varepsilon\to 0^+$ inside the first integral because we integrate a continuous function over the compact interval $[s_0,s_1]$. This equation now shows that $\psi\in\AC(I)$ and together with \eqref{Eq_Integral_derivative_2} we conclude that the derivative is almost everywhere given by
\begin{align*}
\frac{d}{ds}\psi(s)&=-b'(s)e^{ia(s)b(s)^2}f(s,b(s))+\mathcal{I}_{a(s),b(s)}(g(s,\cdot\,)) \\
&=-b'(s)e^{ia(s)b(s)^2}f(s,b(s))+\lim\limits_{\varepsilon\to 0^+}\int_{b(s)}^\infty e^{-\varepsilon y^2}\frac{d}{ds}\big(e^{ia(s)y^2}f(s,y)\big)dy. \qedhere
\end{align*}
\end{proof}

The next result is of a similar nature as the previous theorem. Here we provide conditions on a function $f:\mathcal{U}\times[b,\infty)\to\mathbb{C}$ such that
the oscillatory integral $\mathcal{I}_{a,b}(f(z,\cdot\,))$ is holomorphic as a function in $z$.

\begin{thm}
Let $\mathcal{U}\subseteq\mathbb{C}$ be open, $a\in\mathbb{R}\setminus\{0\}$, and $b>0$. Assume that $f:\mathcal{U}\times[b,\infty)\to\mathbb{C}$ satisfies the following conditions (i) and (ii).

\begin{enumerate}
\item[(i)] $f(\,\cdot\,,y)$ is holomorphic on $\mathcal{U}$ for every $y\in[b,\infty)$;
\item[(ii)] There exist $n\in\mathbb{N}$, $\alpha\geq 0$ with $n>\alpha+1$, such that $f(z,\cdot\,)\in C_\alpha^n([b,\infty))$ for every $z\in\mathcal{U}$
and the norm $\Vert f(z,\cdot\,)\Vert_{C_\alpha^n([b,\infty))}$ is locally bounded on $\mathcal{U}$.
\end{enumerate}
Then the oscillatory integral
\begin{equation}\label{Eq_Psi_holomorphic}
\psi(z):=\mathcal{I}_{a,b}(f(z,\cdot\,))=\lim\limits_{\varepsilon\to 0^+}I_{a,b}^{\varepsilon,y_0}(f(z,\cdot))=\lim\limits_{\varepsilon\to 0^+}\int_b^\infty e^{-\varepsilon(y-y_0)^2}e^{iay^2}f(z,y)dy
\end{equation}
exists for every $z\in\mathcal{U}$, is independent of $y_0\in\mathbb{R}$, and defines a holomorphic function $\psi$.
\end{thm}

\begin{proof}
Assumption (ii) and Theorem~\ref{thm_Integral} show that the limit in \eqref{Eq_Psi_holomorphic} exists for every $z\in\mathcal{U}$, is independent of $y_0\in\mathbb{R}$, and for every $0<\varepsilon\leq 1$ one has the estimate
\begin{equation}\label{Eq_Integral_holomorphic_1}
|I_{a,b}^{\varepsilon,y_0}(f(z,\cdot\,))|\leq D_{a,b,n,\alpha}\Vert f(z,\cdot\,)\Vert_{C_\alpha^n([b,\infty))}.
\end{equation}
In order to show that the resulting function $\psi$ is holomorphic, we integrate $\psi$ along the boundary of an arbitrary closed triangle $\Delta\subseteq\mathcal{U}$, i.e. we consider
\begin{align}
\oint_{\partial\Delta}\psi(z)dz&=\oint_{\partial\Delta}\lim\limits_{\varepsilon\to 0^+}I_{a,b}^{\varepsilon,y_0}(f(z,\cdot\,))dz=\lim\limits_{\varepsilon\to 0^+}\oint_{\partial\Delta}I_{a,b}^{\varepsilon,y_0}(f(z,\cdot\,))dz \notag \\
&=\lim\limits_{\varepsilon\to 0^+}\oint_{\partial\Delta}\int_b^\infty e^{-\varepsilon(y-y_0)^2}e^{iay^2}f(z,y)dydz, \label{Eq_Integral_holomorphic_2}
\end{align}
where in the second equality we were allowed to interchange the limit and integral since the right hand side of the $\varepsilon$-uniform estimate \eqref{Eq_Integral_holomorphic_1} is integrable over the compact boundary $\partial\Delta$ by assumption (ii). Next, for every $\varepsilon>0$, using \eqref{Eq_fk_estimate} we can estimate
\begin{equation*}
\big|e^{-\varepsilon(y-y_0)^2}e^{iay^2}f(z,y)\big|\leq e^{-\varepsilon(y-y_0)^2}\Vert f(z,\cdot)\Vert_{C_\alpha^n([b,\infty))}y^{n+\alpha},\qquad z\in\partial\Delta,\,y\in[b,\infty).
\end{equation*}
Since the right hand side is integrable over $\partial\Delta\times[b,\infty)$ we are also allowed to interchange the order of integration in \eqref{Eq_Integral_holomorphic_2}. Hence, we conclude
\begin{equation*}
\oint_{\partial\Delta}\psi(z)dz=\lim\limits_{\varepsilon\to 0^+}\int_b^\infty e^{-\varepsilon(y-y_0)^2}e^{iay^2}\bigg(\oint_{\partial\Delta}f(z,y)dz\bigg)dy=0
\end{equation*}
where we have used $\oint_{\partial\Delta}f(z,y)dz=0$ as $f(\,\cdot\,,y)$ is holomorphic on $\mathcal{U}$ by assumption (i). Now Morera's theorem implies that the function $\psi$ is holomorphic on $\mathcal{U}$.
\end{proof}

\section{The one dimensional time dependent Schrödinger equation}\label{sec_Time_dependent_Schroedinger_equation}

In this section we will use the general theory from Section~\ref{sec_Oscillatory_integrals} to express the solution of the time dependent Schrödinger equation 
\begin{equation}\label{Eq_Schroedinger_G}
\begin{split}
i\frac{\partial}{\partial t}\Psi(t,x)&=\Big(-\frac{\partial^2}{\partial x^2}+V(t,x)\Big)\Psi(t,x),\qquad {\rm f.a.e.}\;t\in(0,T),\,x\in\mathbb{R}, \\
\Psi(0,x)&=F(x),\hspace{3.9cm} x\in\mathbb{R},
\end{split}
\end{equation}
for some measurable potential $V:(0,T)\times\mathbb{R}\to\mathbb{C}$, $T\in (0,\infty]$, via an oscillatory integral involving the Green's function $G$ and the initial condition $F$; cf. Theorem~\ref{thm_Psi} below. For our purposes it is convenient to use the space
\begin{equation*}
\AC_{1,2}((0,T)\times\mathbb{R}):=\Set{\Psi:(0,T)\times\mathbb{R}\to\mathbb{C} | \hspace{-0.2cm}\begin{array}{l} \Psi(\,\cdot\,,x)\in\AC((0,T)),\,x\in\mathbb{R} \\ \Psi(t,\cdot\,),\frac{\partial\Psi}{\partial x}(t,\cdot\,)\in\AC(\mathbb{R}),\,t\in(0,T) \end{array} \hspace{-0.2cm}},
\end{equation*}
which was also considered in  \cite{ABCS22} (and in a similar form in \cite{S22}). In fact, in \cite{ABCS22,S22} the solution of \eqref{Eq_Schroedinger_G} is also represented as an integral of the form \eqref{Eq_Psi}, but only for holomorphic Green's functions and initial conditions. Theorem~\ref{thm_Psi} and Theorem~\ref{thm_Psi_continuous_dependency} below can be viewed as improvements in the sense that the regularity requirements on the Green's function in Assumption~\ref{ass_Greensfunction} and the initial condition in Theorem~\ref{thm_Psi} and Theorem~\ref{thm_Psi_continuous_dependency} are less restrictive.

\begin{ass}\label{ass_Greensfunction}
Let $T\in(0,\infty]$ and suppose that the Green's function
\begin{equation*}
G:(0,T)\times\mathbb{R}\times\mathbb{R}\to\mathbb{C},
\end{equation*}
satisfies the following conditions (i)--(iii).

\begin{enumerate}
\item[(i)] For every $y\in\mathbb{R}$ the function $G(\,\cdot\,,\,\cdot\,,y)\in\AC_{1,2}((0,T)\times\mathbb{R})$ is a solution of the time dependent Schrödinger equation
in \eqref{Eq_Schroedinger_G};

\item[(ii)] For every $x_0>0$ and $\varphi\in C^1([-x_0,x_0])$ one has
\begin{equation}\label{Eq_Initial_G}
\lim\limits_{t\to 0^+}\int_{-x_0}^{x_0}G(t,x,y)\varphi(y)dy=\varphi(x),\qquad x\in(-x_0,x_0);
\end{equation}

\item[(iii)] For some $a\in\AC((0,T))$ with $a(t)>0$, $t\in (0,T)$, and $\lim_{t\to 0^+}a(t)=\infty$ one has the decomposition
\begin{equation}\label{Eq_G_decomposition}
G(t,x,y)=e^{ia(t)(y-x)^2}\widetilde{G}(t,x,y),\qquad t\in(0,T),\,x,y\in\mathbb{R},
\end{equation}
with a function $\widetilde{G}$, which for some $\alpha\geq 0$, $n\in\mathbb{N}$ with $n>\alpha+3$ satisfies
\begin{subequations}\label{Eq_G_estimates}
\begin{align}
\widetilde{G}(t,x,\cdot\,)&\in C^n(\mathbb{R},r^\alpha), \label{Eq_G_estimate} \\
\widetilde{G}_x(t,x,\cdot\,)&\in C^n(\mathbb{R},r^{\alpha+1}), \label{Eq_Gx_estimate} \\
\widetilde{G}_{xx}(t,x,\cdot\,),\widetilde{G}_t(t,x,\cdot\,)&\in C^n(\mathbb{R},r^{\alpha+2}). \label{Eq_Gxx_Gt_estimate}
\end{align}
\end{subequations}
Moreover, the norms $\Vert\widetilde{G}(t,x,\cdot\,)\Vert_{C^n(\mathbb{R},r^\alpha)}$, $\Vert\widetilde{G}_x(t,x,\cdot\,)\Vert_{C^n(\mathbb{R},r^{\alpha+1})}$, $\Vert\widetilde{G}_{xx}(t,x,\cdot\,)\Vert_{C^n(\mathbb{R},r^{\alpha+2})}$, and $\Vert\widetilde{G}_t(t,x,\cdot\,)\Vert_{C^n(\mathbb{R},r^{\alpha+2})}$ are locally bounded on $(0,T)\times\mathbb{R}$ and 
\begin{equation}\label{Eq_Gtilde_initial}
\lim\limits_{t\to 0^+}\frac{\Vert \widetilde{G}(t,x,\cdot\,)\Vert_{C^n(\mathbb{R},r^\alpha)}}{a(t)}=0.
\end{equation}
\end{enumerate}
\end{ass}

\noindent The next theorem is the main result of this section.

\begin{thm}\label{thm_Psi}
Let $T\in(0,\infty]$ and $G:(0,T)\times\mathbb{R}\times\mathbb{R}\to\mathbb{C}$ be as in 
Assumption~\ref{ass_Greensfunction}, where \eqref{Eq_G_estimates} holds for some $\alpha\geq 0$, $n\in\mathbb{N}$ with $n>\alpha+3$. Then for every $F\in C^n(\mathbb{R},r^\beta)$, $0\leq\beta<n-\alpha-3$, the limit
\begin{equation}\label{Eq_Psi}
\Psi(t,x):=\lim\limits_{\varepsilon\to 0^+}\int_\mathbb{R}e^{-\varepsilon y^2}G(t,x,y)F(y)dy,\qquad t\in(0,T),\,x\in\mathbb{R},
\end{equation}
exists and yields a solution $\Psi\in\AC_{1,2}((0,T)\times\mathbb{R})$ of the Schrödinger equation \eqref{Eq_Schroedinger_G}.
\end{thm}

\begin{proof}
As a preparation for the main part of the proof, we set
\begin{equation}\label{Eq_Ghat}
\widehat{G}(t,x,y):=e^{-ia(t)y^2}G(t,x,y)=e^{ia(t)(x^2-2xy)}\widetilde{G}(t,x,y),
\end{equation}
which is, beside \eqref{Eq_G_decomposition}, another decomposition of the function $G$. We will prove that
\begin{subequations}\label{Eq_Ghat_estimates}
\begin{align}
\widehat{G}(t,x,\cdot\,)&\in C^n(\mathbb{R},r^\alpha), \label{Eq_Ghat_estimate} \\
\widehat{G}_x(t,x,\cdot\,)&\in C^n(\mathbb{R},r^{\alpha+1}) \label{Eq_Ghatx_estimate} \\
\widehat{G}_t(t,x,\cdot\,),\widehat{G}_{xx}(t,x,\cdot\,)&\in C^n(\mathbb{R},r^{\alpha+2}), \label{Eq_Ghatxx_Ghatt_estimate}
\end{align}
\end{subequations}
where the corresponding norms $\Vert\widehat{G}(t,x,\cdot\,)\Vert_{C^n(\mathbb{R},r^\alpha)}$, $\Vert\widehat{G}_x(t,x,\cdot\,)\Vert_{C^n(\mathbb{R},r^{\alpha+1})}$, $\Vert\widehat{G}_t(t,x,\cdot\,)\Vert_{C^n(\mathbb{R},r^{\alpha+2})}$, and $\Vert\widehat{G}_{xx}(t,x,\cdot\,)\Vert_{C^n(\mathbb{R},r^{\alpha+2})}$ are locally bounded on $(0,T)\times\mathbb{R}$. \medskip

Since $\widetilde{G}(t,x,\cdot\,)\in C^n(\mathbb{R},r^\alpha)$ by \eqref{Eq_G_estimate} we also have $\widehat{G}(t,x,\cdot\,)\in C^n(\mathbb{R},r^\alpha)$ by Lemma~\ref{lem_ymf_expf}~(ii), with norm bounded by
\begin{equation*}
\Vert\widehat{G}(t,x,\cdot\,)\Vert_{C^n(\mathbb{R},r^\alpha)}\leq(1+2a(t)|x|)^n\Vert\widetilde{G}(t,x,\cdot\,)\Vert_{C^n(\mathbb{R},r^\alpha)};
\end{equation*}
thus \eqref{Eq_Ghat_estimate} is clear. Since $a$ is continuous and $\Vert\widetilde{G}(t,x,\cdot\,)\Vert_{C^n(\mathbb{R},r^\alpha)}$ is assumed to be locally bounded on $(0,T)\times\mathbb{R}$, also $\Vert\widehat{G}(t,x,\cdot\,)\Vert_{C^n(\mathbb{R},r^\alpha)}$ is locally bounded on $(0,T)\times\mathbb{R}$. For the proof of \eqref{Eq_Ghatx_estimate}, we first compute the $x$-derivative of $\widehat{G}$ in \eqref{Eq_Ghat}, namely
\begin{equation*}
\widehat{G}_x(t,x,y)=2ia(t)(x-y)\widehat{G}(t,x,y)+e^{ia(t)(x^2-2xy)}\widetilde{G}_x(t,x,y).
\end{equation*}
Since we have already shown $\widehat{G}(t,x,\cdot\,)\in C^n(\mathbb{R},r^\alpha)$ and since $\widetilde{G}_x(t,x,\cdot\,)\in C^n(\mathbb{R},r^{\alpha+1})$ by assumption \eqref{Eq_Gx_estimate}, it follows from Lemma~\ref{lem_ymf_expf} that $\widehat{G}_x(t,x,\cdot\,)\in C^n(\mathbb{R},r^{\alpha+1})$, with 
\begin{align*}
\Vert\widehat{G}_x(t,&x,\cdot\,)\Vert_{C^n(\mathbb{R},r^{\alpha+1})} \\
&\leq 2a(t)\Vert(x-y)\widehat{G}(t,x,\cdot\,)\Vert_{C^n(\mathbb{R},r^{\alpha+1})}+\Vert e^{-2ia(t)xy}\widetilde{G}_x(t,x,\cdot\,)\Vert_{C^n(\mathbb{R},r^{\alpha+1})} \\
&\leq 2a(t)(|x|+2^n)\Vert\widehat{G}(t,x,\cdot\,)\Vert_{C^n(\mathbb{R},r^\alpha)}+(1+2a(t)|x|)^n\Vert\widetilde{G}_x(t,x,\cdot\,)\Vert_{C^n(\mathbb{R},r^{\alpha+1})}.
\end{align*}
In particular, this norm is locally bounded on $(0,T)\times\mathbb{R}$. For the proof of \eqref{Eq_Ghatxx_Ghatt_estimate} we compute the $t$- and the second $x$-derivative
\begin{align*}
\widehat{G}_t(t,x,y)&=ia'(t)(x^2-2xy)\widehat{G}(t,x,y)+e^{ia(t)(x^2-2xy)}\widetilde{G}_t(t,x,y), \\
\widehat{G}_{xx}(t,x,y)&=2a(t)\big(2a(t)(x-y)^2+i\big)\widehat{G}(t,x,y)+4ia(t)(x-y)\widehat{G}_x(t,x,y) \\
&\quad+e^{ia(t)(x^2-2xy)}\widetilde{G}_{xx}(t,x,y).
\end{align*}
Similarly, also these functions satisfy $\widehat{G}_t(t,x,\cdot\,),\widehat{G}_{xx}(t,x,\cdot\,)\in C^n(\mathbb{R},r^{\alpha+2})$ by Lemma~\ref{lem_ymf_expf}, with norms which are locally bounded on $(0,T)\times\mathbb{R}$. Hence we have proven all three properties \eqref{Eq_Ghat_estimates}. \medskip

For the main part of the proof we first fix some arbitrary $b>0$, and note that it is sufficient to verify the assertions for $x\in(-b,b)$ only. Let us split up for every $\varepsilon>0$ the integral in \eqref{Eq_Psi} into
\begin{subequations}
\begin{align}
\int_\mathbb{R}e^{-\varepsilon y^2}G(t,x,y)F(y)dy&=\int_b^\infty e^{-\varepsilon y^2}G(t,x,y)F(y)dy \label{Eq_Schroedinger_solution_13} \\
&\quad+\int_{-b}^be^{-\varepsilon y^2}G(t,x,y)F(y)dy \label{Eq_Schroedinger_solution_14} \\
&\quad+\int_{-\infty}^{-b}e^{-\varepsilon y^2}G(t,x,y)F(y)dy. \label{Eq_Schroedinger_solution_15}
\end{align}
\end{subequations}
The proof will now be done in four steps. In first three steps we prove that the limits $\varepsilon\to 0^+$ in \eqref{Eq_Schroedinger_solution_13}, \eqref{Eq_Schroedinger_solution_14} and \eqref{Eq_Schroedinger_solution_15} exist and are solutions of the Schrödinger equation in \eqref{Eq_Schroedinger_G}. In Step 4 we will verify the initial condition in \eqref{Eq_Schroedinger_G}. \medskip

\noindent {\it Step 1.} In the first step we will use Theorem~\ref{thm_Integral_derivative} to show that the limit
\begin{equation}\label{Eq_Schroedinger_solution_1}
\Psi_1(t,x):=\lim_{\varepsilon\to 0^+}\int_b^\infty e^{-\varepsilon y^2}G(t,x,y)F(y)dy,\qquad t\in(0,T),\,x\in\mathbb{R},
\end{equation}
exists, and that we are allowed to carry the $t$- and $x$-derivatives inside the integral. Plugging in the function $\widehat{G}$ from \eqref{Eq_Ghat}, we rewrite
\begin{equation}\label{Eq_Schroedinger_solution_11}
\int_b^\infty e^{-\varepsilon y^2}G(t,x,y)F(y)dy=\int_b^\infty e^{(ia(t)-\varepsilon)y^2}\widehat{G}(t,x,y)F(y)dy=I_{a(t),b}^{\varepsilon,0}(f(t,x,\cdot\,)),
\end{equation}
using the function
\begin{equation}\label{Eq_Schroedinger_solution_3}
f(t,x,y):=\widehat{G}(t,x,y)F(y).
\end{equation}
Let us now fix $x\in(-b,b)$ and verify that the function $f(\,\cdot\,,x,\cdot\,)$ satisfies conditions (i) and (ii) in Theorem~\ref{thm_Integral_derivative}. In fact, for (i) recall that $G(\,\cdot\,,\cdot\,,y)\in\AC_{1,2}((0,T)\times\mathbb{R})$ by Assumption~\ref{ass_Greensfunction}~(i), so that $\widehat{G}(\,\cdot\,,x,y)\in\AC((0,T))$ and hence also $f(\,\cdot\,,x,y)\in\AC((0,T))$. For condition (ii) in Theorem~\ref{thm_Integral_derivative}, note that $\widehat{G}(t,x,\cdot\,)\in C^n(\mathbb{R},r^\alpha)$ and $F\in C^n(\mathbb{R},r^\beta)$. It then follows from Proposition~\ref{prop_Cn_product} and Lemma~\ref{lem_Cnr_Cnalpha} that $f(t,x,\cdot\,)\in C^n_{\alpha+\beta}([b,\infty))$, with norm bounded by
\begin{align}
\Vert f(t,x,\cdot\,)\Vert_{C^n_{\alpha+\beta}([b,\infty))}&\leq\Big(1+\frac{1}{b}\Big)^{n+\alpha+\beta}\Vert f(t,x,\cdot\,)\Vert_{C^n(\mathbb{R},r^{\alpha+\beta})} \notag \\
&\leq\Big(1+\frac{1}{b}\Big)^{n+\alpha+\beta}2^n\Vert\widehat{G}(t,x,\cdot\,)\Vert_{C^n(\mathbb{R},r^\alpha)}\Vert F\Vert_{C^n(\mathbb{R},r^\beta)}. \label{Eq_Schroedinger_solution_4}
\end{align}
Since $\Vert\widehat{G}(t,x,\cdot\,)\Vert_{C^n(\mathbb{R},r^\alpha)}$ is locally bounded on $(0,T)\times\mathbb{R}$, the same is true for $\Vert f(t,x,\cdot\,)\Vert_{C_{\alpha+\beta}^n([b,\infty))}$. Similar estimates as in \eqref{Eq_Schroedinger_solution_4} show that also
\begin{equation*}
\frac{\partial f}{\partial t}(t,x,\cdot\,)=\widehat{G}_t(t,x,\cdot\,)F(\,\cdot\,)\in C_{\alpha+\beta+2}^n([b,\infty)),
\end{equation*}
with locally bounded norm $\Vert\frac{\partial f}{\partial t}(t,x,\cdot\,)\Vert_{C_{\alpha+\beta+2}^n([b,\infty))}$. Since $n>\alpha+\beta+3$, the assumptions of Theorem~\ref{thm_Integral_derivative} are satisfied, and it follows that the limit \eqref{Eq_Schroedinger_solution_1} exists, namely
\begin{align}
\mathcal{I}_{a(t),b}(f(t,x,\cdot\,))&=\lim\limits_{\varepsilon\to 0^+}I_{a(t),b}^{\varepsilon,0}(f(t,x,\cdot\,)) \notag \\
&=\lim\limits_{\varepsilon\to 0^+}\int_b^\infty e^{-\varepsilon y^2}G(t,x,y)F(y)dy \label{Eq_Schroedinger_solution_6} \\
&=\Psi_1(t,x), \notag
\end{align}
and leads to a function $\Psi_1(\,\cdot\,,x)\in\AC((0,T))$ with derivative 
\begin{equation}\label{Eq_Schroedinger_solution_8}
\frac{\partial}{\partial t}\Psi_1(t,x)
=\lim\limits_{\varepsilon\to 0^+}\int_b^\infty e^{-\varepsilon y^2}G_t(t,x,y)F(y)dy.
\end{equation}
In order to show that $\Psi_1$ is differentiable with respect to $x$, we will again verify conditions (i) and (ii) in Theorem~\ref{thm_Integral_derivative}, this time for $f(t,\cdot\,,\cdot\,)$ in \eqref{Eq_Schroedinger_solution_3} with $t\in(0,T)$ fixed. Condition~(i) in this context reads as $f(t,\cdot\,,y)\in\AC(\mathbb{R})$, and holds since $G(\,\cdot\,,\cdot\,,y)\in\AC_{1,2}((0,T)\times\mathbb{R})$ and consequently $\widehat{G}(t,\cdot\,,y)\in\AC(\mathbb{R})$. We have already shown that $f(t,x,\cdot\,)\in C_{\alpha+\beta}^n([b,\infty))$, and the corresponding norm bound \eqref{Eq_Schroedinger_solution_4} is also locally integrable in $x$. Similar estimates as in \eqref{Eq_Schroedinger_solution_4} show that the $x$-derivative satisfies
\begin{equation*}
\frac{\partial f}{\partial x}(t,x,\cdot\,)=\widehat{G}_x(t,x,\cdot\,)F(\,\cdot\,)\in C_{\alpha+\beta+1}^n([b,\infty)),
\end{equation*}
with locally bounded norm $\Vert\frac{\partial f}{\partial x}(t,x,\cdot\,)\Vert_{C_{\alpha+\beta+1}^n([b,\infty))}$. Hence it follows from Theorem~\ref{thm_Integral_derivative} that $\Psi_1(t,\cdot\,)\in \AC(\mathbb R)$, with $x$-derivative given by
\begin{equation}\label{Eq_Schroedinger_solution_7}
\frac{\partial}{\partial x}\Psi_1(t,x)
=\lim\limits_{\varepsilon\to 0^+}\int_b^\infty e^{-\varepsilon y^2}G_x(t,x,y)F(y)dy.
\end{equation}
For the second $x$-derivative, we have to differentiate \eqref{Eq_Schroedinger_solution_7} once more and in a similar way as above we verify $\frac{\partial\Psi_1}{\partial x}(t,\cdot\,)\in\AC(\mathbb{R})$ and obtain
\begin{equation}\label{Eq_Schroedinger_solution_10}
\frac{\partial^2}{\partial x^2}\Psi_1(t,x)=\lim\limits_{\varepsilon\to 0^+}\int_b^\infty e^{-\varepsilon y^2}G_{xx}(t,x,y)F(y)dy.
\end{equation}
Combining now \eqref{Eq_Schroedinger_solution_8} and \eqref{Eq_Schroedinger_solution_10}, and the fact that $G$ is a solution of the Schrödinger equation \eqref{Eq_Schroedinger_G}, shows $\Psi_1\in\AC_{1,2}((0,T)\times\mathbb{R})$ and
\begin{equation*}
i\frac{\partial}{\partial t}\Psi_1(t,x)=\Big(-\frac{\partial^2}{\partial x^2}+V(t,x)\Big)\Psi_1(t,x)\qquad{\rm f.a.e.}\;t\in(0,T),\,x\in\mathbb{R}.
\end{equation*}
{\it Step 2.} In the second step we consider the integral \eqref{Eq_Schroedinger_solution_15}. If we substitute $y\to -y$, we conclude in the same way as in Step 1 that
\begin{equation}\label{Eq_Schroedinger_solution_9}
\Psi_{-1}(t,x):=\lim\limits_{\varepsilon\to 0^+}\int_{-\infty}^{-b}e^{-\varepsilon y^2}G(t,x,y)F(y)dy=\lim\limits_{\varepsilon\to 0^+}\int_b^\infty e^{-\varepsilon y^2}G(t,x,-y)F(-y)dy
\end{equation}
exists, and that $\Psi_{-1}\in\AC_{1,2}((0,T)\times\mathbb{R})$ satisfies
\begin{equation*}
i\frac{\partial}{\partial t}\Psi_{-1}(t,x)=\Big(-\frac{\partial^2}{\partial x^2}+V(t,x)\Big)\Psi_{-1}(t,x)\qquad{\rm f.a.e.}\;t\in(0,T),\,x\in\mathbb{R}.
\end{equation*}
{\it Step 3.} For the integral \eqref{Eq_Schroedinger_solution_14}, we note that the integrand is continuous and the integration interval $[-b,b]$ is compact. Hence the limit
\begin{equation}\label{Eq_Schroedinger_solution_2}
\Psi_0(t,x):=\lim\limits_{\varepsilon\to 0^+}\int_{-b}^be^{-\varepsilon y^2}G(t,x,y)F(y)dy=\int_{-b}^bG(t,x,y)F(y)dy
\end{equation}
exists. Next, the $t$-derivative of the Green's function, in the decomposition \eqref{Eq_Ghat}, can be estimated using Lemma~\ref{lem_ymf_expf}~(i) by
\begin{align*}
\Vert G_t(t,x,\cdot\,)\Vert_{C^n(\mathbb{R},r^{\alpha+2})}&=\big\Vert ia'(t)y^2\widehat{G}(t,x,\cdot\,)+\widehat{G}_t(t,x,\cdot\,)\big\Vert_{C^n(\mathbb{R},r^{\alpha+2})} \\
&\leq 3^n|a'(t)|\Vert\widehat{G}(t,x,\cdot\,)\Vert_{C^n(\mathbb{R},r^\alpha)}+\Vert\widehat{G}_t(t,x,\cdot\,)\Vert_{C^n(\mathbb{R},r^{\alpha+2})}.
\end{align*}
Since $a'\in L_\loc^1((0,T))$ because of $a\in\AC((0,T))$, and since both norms $\Vert\widehat{G}(t,x,\cdot\,)\Vert_{C^n(\mathbb{R},r^\alpha)}$ and $\Vert\widehat{G}_t(t,x,\cdot\,)\Vert_{C^n(\mathbb{R},r^{\alpha+2})}$ are locally bounded on $(0,T)\times\mathbb{R}$ by \eqref{Eq_G_estimate} and \eqref{Eq_Gxx_Gt_estimate}, the right hand side of this inequality is integrable over $[t_0,t_1]\times[-b,b]$ for every compact interval $[t_0,t_1]\subseteq(0,T)$. Interchanging the order of integration then gives
\begin{align*}
\int_{t_0}^{t_1}\bigg(\int_{-b}^bG_t(t,x,y)F(y)dy\bigg)dt&=\int_{-b}^b\bigg(\int_{t_0}^{t_1}G_t(t,x,y)dt\bigg)F(y)dy \\
&=\int_{-b}^bG(t,x,y)F(y)dy\Big|_{t=t_0}^{t_1}\\
&=\Psi_0(t,x)\big|_{t=t_0}^{t_1}.
\end{align*}
This proves that $\Psi_0(\,\cdot\,,x)\in\AC((0,T))$ with $t$-derivative
\begin{equation}\label{Eq_Schroedinger_solution_16}
\frac{\partial}{\partial t}\Psi_0(t,x)=\int_{-b}^bG_t(t,x,y)F(y)dy.
\end{equation}
In the same way one also verifies $\Psi_0(t,\cdot\,),\frac{\partial\Psi_0}{\partial x}(t,\cdot\,)\in\AC(\mathbb{R})$ with second $x$-derivative
\begin{equation}\label{Eq_Schroedinger_solution_12}
\frac{\partial^2}{\partial x^2}\Psi_0(t,x)=\int_{-b}^bG_{xx}(t,x,y)F(y)dy.
\end{equation}
Combining now \eqref{Eq_Schroedinger_solution_16} and \eqref{Eq_Schroedinger_solution_12} shows $\Psi_0\in\AC_{1,2}((0,T)\times\mathbb{R})$ and
\begin{equation*}
i\frac{\partial}{\partial t}\Psi_0(t,x)=\Big(-\frac{\partial^2}{\partial x^2}+V(t,x)\Big)\Psi_0(t,x)\qquad{\rm f.a.e.}\;t\in(0,T),\,x\in\mathbb{R}.
\end{equation*}
From the considerations in Step 1, 2, and 3 above, it now follows that the limit
\begin{equation}\label{Eq_Schroedinger_solution_22}
\Psi(t,x)=\lim\limits_{\varepsilon\to 0^+}\int_\mathbb{R}e^{-\varepsilon y^2}G(t,x,y)F(y)dy=\Psi_1(t,x)+\Psi_0(t,x)+\Psi_{-1}(t,x)
\end{equation}
in \eqref{Eq_Psi} exists, and $\Psi\in\AC_{1,2}((0,T)\times\mathbb{R})$ is a solution of the Schrödinger equation in \eqref{Eq_Schroedinger_G}. \medskip

\noindent {\it Step 4.} Now the initial condition in \eqref{Eq_Schroedinger_G} will be verified, again for the three functions $\Psi_1$, $\Psi_0$ and $\Psi_{-1}$ separately. We start with $\Psi_1$ in \eqref{Eq_Schroedinger_solution_6} and note that by Theorem~\ref{thm_Integral} with $y_0=x$ and \eqref{Eq_Ghat} the oscillatory integral can also be written in the form
\begin{align}
\Psi_1(t,x)&=\lim\limits_{\varepsilon\to 0^+}\int_b^\infty e^{(ia(t)-\varepsilon)(y-x)^2}\widetilde{G}(t,x,y)F(y)dy \notag \\
&=\lim\limits_{\varepsilon\to 0^+}\int_{b-x}^\infty e^{(ia(t)-\varepsilon)y^2}\widetilde{G}(t,x,y+x)F(y+x)dy \label{Eq_Schroedinger_solution_17} \\
&=\mathcal{I}_{a(t),b-x}(g(t,x,\cdot\,)), \notag
\end{align}
using the function
\begin{equation*}
g(t,x,y):=\widetilde{G}(t,x,y+x)F(y+x).
\end{equation*}
Note that $g(t,x,\cdot\,)\in C_{\alpha+\beta}^n([b-x,\infty))$ (this is clear from the norm estimate below) and since we have assumed $x\in(-b,b)$ the lower integration bound $b-x>0$ in $\mathcal{I}_{a(t),b-x}$ is positive. From $\widetilde{G}(t,x,\cdot\,)\in C^n(\mathbb{R},r^\alpha)$ and $F\in C^n(\mathbb{R},r^\beta)$ we conclude together with Lemma~\ref{lem_Cnr_Cnalpha}, Proposition~\ref{prop_Cn_product}, and Lemma~\ref{lem_Cn_shifted} that 
\begin{align*}
\Vert g(t,x,\cdot\,)\Vert_{C^n_{\alpha+\beta}([b-x,\infty))}&\leq\Big(1+\frac{1}{b-x}\Big)^{n+\alpha+\beta}\Vert g(t,x,\cdot\,)\Vert_{C^n(\mathbb{R},r^{\alpha+\beta})} \\
&\leq\Big(1+\frac{1}{b-x}\Big)^{n+\alpha+\beta}2^n\Vert\widetilde{G}(t,x,x+\,\cdot\,)\Vert_{C^n(\mathbb{R},r^\alpha)}\Vert F(x+\,\cdot\,)\Vert_{C^n(\mathbb{R},r^\beta)} \\
&\leq\Big(1+\frac{1}{b-x}\Big)^{n+\alpha+\beta}2^n(1+|x|)^{\alpha+\beta}\Vert\widetilde{G}(t,x,\cdot\,)\Vert_{C^n(\mathbb{R},r^\alpha)}\Vert F\Vert_{C^n(\mathbb{R},r^\beta)}.
\end{align*}
With \eqref{Eq_D} we can now estimate the function $\Psi_1$ in \eqref{Eq_Schroedinger_solution_17} by
\begin{equation*}
|\Psi_1(t,x)|\leq D_{a(t),b-x,n,\alpha+\beta}\Big(1+\frac{1}{b-x}\Big)^{n+\alpha+\beta}2^n(1+|x|)^{\alpha+\beta}\Vert\widetilde{G}(t,x,\cdot\,)\Vert_{C^n(\mathbb{R},r^\alpha)}\Vert F\Vert_{C^n(\mathbb{R},r^\beta)}.
\end{equation*}
Since $\lim_{t\to\infty}a(t)=\infty$ we have $\limsup_{t\to\infty}a(t)D_{a(t),b-x,n,\alpha,\beta}<\infty$ by \eqref{Eq_D_asymptotics}. Together with the assumption \eqref{Eq_Gtilde_initial} we then conclude 
\begin{equation*}
\lim\limits_{t\to 0^+}\Psi_1(t,x)=0.
\end{equation*}
For the same reason one also has
\begin{equation*}
\lim_{t\to 0^+}\Psi_{-1}(t,x)=0.
\end{equation*}
Eventually, it follows from \eqref{Eq_Initial_G} that
\begin{equation*}
\lim\limits_{t\to 0^+}\Psi_0(t,x)=\lim\limits_{t\to 0^+}\int_{-b}^bG(t,x,y)F(y)dy=F(x),
\end{equation*}
and hence the initial condition in \eqref{Eq_Schroedinger_G} is clear from the decomposition \eqref{Eq_Schroedinger_solution_22}.
\end{proof}

The next result complements Theorem~\ref{thm_Psi} and shows that the solution of the time dependent Schrödinger equation depends continuously on the initial condition. In order to emphasize the initial condition we will denote the solution of \eqref{Eq_Schroedinger_G} by $\Psi(\,\cdot\,,\cdot\,;F)$.

\begin{thm}\label{thm_Psi_continuous_dependency}
Let $T\in(0,\infty]$ and $G:(0,T)\times\mathbb{R}\times\mathbb{R}\to\mathbb{C}$ be as in Assumption~\ref{ass_Greensfunction}, where \eqref{Eq_G_estimates} holds for some $\alpha\geq 0$, $n\in\mathbb{N}$ with $n>\alpha+3$. Assume that $F,(F_m)_m\in C^n(\mathbb{R},r^\beta)$, $0\leq\beta<n-\alpha-3$, are such that
\begin{equation*}
\lim\limits_{m\to\infty}\Vert F-F_m\Vert_{C^n(\mathbb{R},r^\beta)}=0.
\end{equation*}
Then also the corresponding solutions of \eqref{Eq_Schroedinger_G} satisfy
\begin{equation*}
\lim\limits_{m\to\infty}\Psi(t,x;F_m)=\Psi(t,x;F),
\end{equation*}
where the convergence is locally uniform on $(0,T)\times\mathbb{R}$.
\end{thm}

\begin{proof}
Using the decomposition \eqref{Eq_Schroedinger_solution_22}, it suffices to prove the convergence for the three functions $\Psi_1$, $\Psi_0$ and $\Psi_{-1}$ in \eqref{Eq_Schroedinger_solution_1}, \eqref{Eq_Schroedinger_solution_2} and \eqref{Eq_Schroedinger_solution_9}, respectively. For $\Psi_1$ we write similarly as in  \eqref{Eq_Schroedinger_solution_11}
\begin{align*}
\Psi_1(t,x;F_m)-\Psi_1(t,x;F)&=\lim\limits_{\varepsilon\to 0^+}\int_b^\infty e^{(ia(t)-\varepsilon)y^2}\widehat{G}(t,x,y)(F_m(y)-F(y))dy \\
&=\lim\limits_{\varepsilon\to 0^+} I_{a(t),b}^{\varepsilon,0}(f_m(t,x,\cdot\,)),
\end{align*}
where we have set $f_m(t,x,y):=\widehat{G}(t,x,y)(F_m(y)-F(y))$. Following the same reasoning as in \eqref{Eq_Schroedinger_solution_4}, we obtain the norm bound
\begin{equation*}
\Vert f_m(t,x,\cdot\,)\Vert_{C_{\alpha+\beta}^n([b,\infty))}\leq\Big(1+\frac{1}{b}\Big)^{n+\alpha+\beta}2^n\Vert\widehat{G}(t,x,\cdot\,)\Vert_{C^n(\mathbb{R},r^\alpha)}\Vert F_m-F\Vert_{C^n(\mathbb{R},r^\beta)}.
\end{equation*}
Using the bound \eqref{Eq_D} for the oscillatory integral, we conclude
\begin{align*}
|\Psi_1(t,x;F_m)&-\Psi_1(t,x;F)| \\
&\leq D_{a(t),b,n,\alpha+\beta}\Vert f_m(t,x,\cdot\,)\Vert_{C^n_{\alpha+\beta}([b,\infty))} \\
&\leq D_{a(t),b,n,\alpha+\beta}\Big(1+\frac{1}{b}\Big)^{n+\alpha+\beta}2^n\Vert\widehat{G}(t,x,\cdot\,)\Vert_{C^n(\mathbb{R},r^\alpha)}\Vert F_m-F\Vert_{C^n(\mathbb{R},r^\beta)}\to 0
\end{align*}
as $m\to\infty$. Since the coefficient $D_{a(t),b,n,\alpha+\beta}$ is continuous in $t$ and $\Vert\widehat{G}(t,x,\cdot\,)\Vert_{C^n(\mathbb{R},r^\alpha)}$ is locally bounded on $(0,T)\times\mathbb{R}$ by Assumption~\ref{ass_Greensfunction}~(iii), this limit is locally uniform on $(0,T)\times\mathbb{R}$. The same reasoning shows that also $\Psi_{-1}(\,\cdot\,,\cdot\,;F_m)$ converges locally uniformly on $(0,T)\times\mathbb{R}$ to $\Psi_{-1}(\,\cdot\,,\cdot\,;F)$. For $\Psi_0$ in \eqref{Eq_Schroedinger_solution_2} we again use \eqref{Eq_G_decomposition}, and by the definition of the norm \eqref{Eq_Cnr_norm}, we estimate
\begin{align*}
\big|\Psi_0(t,x;F)-\Psi_0(t,x;F_m)\big|&=\bigg|\int_{-b}^be^{ia(t)(y-x)^2}\widetilde{G}(t,x,y)(F(y)-F_m(y))dy\bigg| \\
&\leq\Vert\widetilde{G}(t,x,\cdot\,)\Vert_{C^n(\mathbb{R},r^\alpha)}\Vert F-F_m\Vert_{C^n(\mathbb{R},r^\beta)}\int_{-b}^br(y)^{\alpha+\beta}dy\to 0
\end{align*}
as $m\to\infty$. Since $\Vert\widetilde{G}(t,x,\cdot\,)\Vert_{C^n(\mathbb{R},r^\alpha)}$ is locally bounded, this convergence is also locally uniform on $(0,T)\times\mathbb{R}$.
\end{proof}

\section{An example: The free particle}\label{sec_The_free_particle}

In this section we illustrate and apply Theorem~\ref{thm_Psi} to the special case of the free particle, i.e. the potential $V(t,x)=0$ in the Schrödinger equation \eqref{Eq_Schroedinger_G}. The wave function $\Psi(t,x)$, defined via the oscillatory integral \eqref{Eq_Psi}, will be computed for the plane wave initial condition $F(y)=e^{i\kappa y}$ in Example~\ref{exam_Plane_wave_initial_conditions} and for monomials $F(y)=y^m$ in Example~\ref{exam_Moments}. We refer to \cite{ABCS22,ACSST13,ACSST17,ACSS19,BCSS14,CSSY22} for many other explicit examples of Green's functions and related considerations in the context of supershifts and superoscillations.

\begin{cor}
For every initial condition $F\in C^n(\mathbb{R},r^\beta)$, with $n\geq 4$, $0\leq\beta<n-3$, the limit
\begin{equation}\label{Eq_Psi_free}
\Psi(t,x):=\frac{1}{2\sqrt{i\pi t}}\lim\limits_{\varepsilon\to 0^+}\int_\mathbb{R}e^{-\varepsilon y^2}e^{\frac{i(y-x)^2}{4t}}F(y)dy,\qquad t\in(0,\infty),\,x\in\mathbb{R},
\end{equation}
exists and defines a solution $\Psi\in\AC_{1,2}((0,T)\times\mathbb{R})$ of the time dependent Schrödinger equation
\begin{align*}
i\frac{\partial}{\partial t}\Psi(t,x)&=-\frac{\partial^2}{\partial x^2}\Psi(t,x), && {\rm f.a.e.}\;t\in(0,\infty),\,x\in\mathbb{R}, \\
\lim\limits_{t\to 0^+}\Psi(t,x)&=F(x), && x\in\mathbb{R}.
\end{align*}
\end{cor}

\begin{proof}
The result follows from Theorem~\ref{thm_Psi}, if we verify that the Green's function
\begin{equation*}
G(t,x,y)=\frac{1}{2\sqrt{i\pi t}}e^{\frac{i(y-x)^2}{4t}},\qquad t\in(0,\infty),\,x,y\in\mathbb{R},
\end{equation*}
satisfies Assumption~\ref{ass_Greensfunction} for $\alpha=0$ and every $n\geq 4$. It is easy to see that for every $y\in\mathbb{R}$ the function $G(\,\cdot\,,\,\cdot\,,y)\in\AC_{1,2}((0,\infty)\times\mathbb{R})$ is a solution of the time dependent Schrödinger equation
\begin{equation*}
i\frac{\partial}{\partial t}G(t,x,y)=-\frac{\partial^2}{\partial x^2}G(t,x,y),\qquad t\in(0,\infty),\,x\in\mathbb{R}.
\end{equation*}
In order to verify the initial condition \eqref{Eq_Initial_G}, we choose $x_0>0$ and $\varphi\in C^1([-x_0,x_0])$. Then for every $x\in(-x_0,x_0)$, we use the error function \eqref{erfi} to write
\begin{align*}
\int_{-x_0}^{x_0}G(t,x,y)\varphi(y)dy&=\frac{1}{2\sqrt{i\pi t}}\int_{-x_0}^{x_0}e^{\frac{i(y-x)^2}{4t}}\varphi(y)dy\\
&=\frac{1}{2}\int_{-x_0}^{x_0}\frac{d}{dy}\erf\Big(\frac{y-x}{2\sqrt{it}}\Big)\varphi(y)dy \\
&=\frac{1}{2}\erf\Big(\frac{y-x}{2\sqrt{it}}\Big)\varphi(y)\Big|_{y=-x_0}^{x_0}-\frac{1}{2}\int_{-x_0}^{x_0}\erf\Big(\frac{y-x}{2\sqrt{it}}\Big)\varphi'(y)dy.
\end{align*}
Using the limit $\lim_{t\to 0^+}\erf\big(\frac{y-x}{2\sqrt{it}}\big)=\sgn(y-x)$, we can now compute the initial condition
\begin{align*}
\lim\limits_{t\to 0^+}\int_{-x_0}^{x_0}G(t,x,y)\varphi(y)dy&=\frac{1}{2}\sgn(y-x)\varphi(y)\Big|_{y=-x_0}^{x_0}-\frac{1}{2}\int_{-x_0}^{x_0}\sgn(y-x)\varphi'(y)dy \\
&=\frac{1}{2}\sgn(y-x)\varphi(y)\Big|_{y=-x_0}^{x_0}+\frac{1}{2}\varphi(y)\Big|_{y=-x_0}^x-\frac{1}{2}\varphi(y)\Big|_{y=x}^{x_0}=\varphi(x),
\end{align*}
so that Assumption~\ref{ass_Greensfunction}~(ii) holds.
Next, the decomposition $G(t,x,y)=e^{ia(t)(y-x)^2}\widetilde{G}(t,x,y)$ in \eqref{Eq_G_decomposition} is satisfied with
\begin{equation*}
a(t)=\frac{1}{4t}\qquad\text{and}\qquad\widetilde{G}(t,x,y)=\frac{1}{2\sqrt{i\pi t}}.
\end{equation*}
Since $\widetilde{G}$ is independent of $y$, it is clear that $\widetilde{G}(t,x,\cdot\,)\in C^n(\mathbb{R},r^0)$, with norm given by
\begin{equation*}
\Vert\widetilde{G}(t,x,\cdot\,)\Vert_{C^n(\mathbb{R},r^0)}=\frac{1}{2\sqrt{\pi t}}.
\end{equation*}
In particular, this norm is locally bounded on $(0,\infty)\times\mathbb{R}$ and one has
\begin{equation*}
\lim\limits_{t\to 0^+}\frac{\Vert\widetilde{G}(t,x,\cdot\,)\Vert_{C^n(\mathbb{R},r^0)}}{a(t)}=\frac{2\sqrt{t}}{\sqrt{\pi}}=0,
\end{equation*}
that is, \eqref{Eq_Gtilde_initial} is satisfied. Moreover, the $t$- and $x$-derivatives of $\widetilde{G}$ are explicitly given by
\begin{equation*}
\widetilde{G}_x(t,x,y)=0,\qquad\widetilde{G}_{xx}(t,x,y)=0,\qquad\text{and}\qquad\widetilde{G}_t(t,x,y)=\frac{-1}{4\sqrt{i\pi}\,t^{\frac{3}{2}}}.
\end{equation*}
Clearly, $\widetilde{G}_x(t,x,\cdot\,)\in C^n(\mathbb{R},r^1)$ and $\widetilde{G}_{xx}(t,x,\cdot\,),\widetilde{G}_t(t,x,\cdot\,)\in C^n(\mathbb{R},r^2)$, with norms given by
\begin{equation*}
\Vert\widetilde{G}_x(t,x,\cdot\,)\Vert_{C^n(\mathbb{R},r^1)}=0,\quad\Vert\widetilde{G}_{xx}(t,x,\cdot\,)\Vert_{C^n(\mathbb{R},r^2)}=0,\quad\text{and}\quad\Vert\widetilde{G}_t(t,x,\cdot\,)\Vert_{C^n(\mathbb{R},r^2)}=\frac{1}{4\sqrt{\pi}\,t^{\frac{3}{2}}},
\end{equation*}
and hence all three norms are locally bounded on $(0,\infty)\times\mathbb{R}$. Therefore, also Assumption~\ref{ass_Greensfunction}~(iii) is satisfied and the assertions follow from Theorem~\ref{thm_Psi}.
\end{proof}

Next we compute the wave function \eqref{Eq_Psi_free} of the free particle with a plane wave initial condition $F(y)=e^{i\kappa y}$ of frequency $\kappa\in\mathbb{R}$.

\begin{exam}\label{exam_Plane_wave_initial_conditions}
Plugging the initial condition $F(y)=e^{i\kappa y}$ into the oscillatory integral \eqref{Eq_Psi_free} gives
\begin{align*}
\Psi(t,x;e^{i\kappa y})&=\frac{1}{2\sqrt{i\pi t}}\lim\limits_{\varepsilon\to 0^+}\int_\mathbb{R}e^{-\varepsilon y^2}e^{\frac{i(y-x)^2}{4t}}e^{i\kappa y}dy \\
&=\frac{e^{\frac{ix^2}{4t}}}{2\sqrt{i\pi t}}\lim\limits_{\varepsilon\to 0^+}\int_\mathbb{R}e^{-(\varepsilon-\frac{i}{4t})y^2+i(\kappa-\frac{x}{2t})y}dy.
\end{align*}
In order to solve this integral for fixed $\varepsilon>0$ we use the explicit value of the absolutely convergent Gauss integral \cite[Equation 7.4.2]{AS72}
\begin{equation*}
\int_\mathbb{R}e^{-\lambda y^2+\mu y}dy=\frac{\sqrt{\pi}}{\sqrt{\lambda}}e^{\frac{\mu^2}{4\lambda}},\qquad\lambda,\mu\in\mathbb{C}\text{ with }\Re(\lambda)>0.
\end{equation*}
Indeed, choosing $\lambda=\varepsilon-\frac{i}{4t}$ and $\mu=i(\kappa-\frac{x}{2t})$ shows that the wave function has the following explicit form
\begin{equation*}
\Psi(t,x;e^{i\kappa y})=\frac{e^{\frac{ix^2}{4t}}}{2\sqrt{i\pi t}}\lim\limits_{\varepsilon\to 0^+}\frac{\sqrt{\pi}}{\sqrt{\varepsilon-\frac{i}{4t}}}e^{\frac{-(\kappa-\frac{x}{2t})^2}{4(\varepsilon-\frac{i}{4t})}}=e^{i\kappa x-i\kappa^2t}.
\end{equation*}
\end{exam}

Now we turn to the moments of the free particle Green's function; here we also refer the reader to \cite{PW22}, where the moments of the Green's function also appear in the context of the time evolution of superoscillations.

\begin{exam}\label{exam_Moments}
As initial conditions we now consider the monomials $F(y)=y^m$. This leads to the moments
\begin{equation*}
g_m(t,x):=\Psi(t,x;y^m)=\frac{1}{2\sqrt{i\pi t}}\lim\limits_{\varepsilon\to 0^+}\int_\mathbb{R}e^{-\varepsilon y^2}e^{\frac{i(y-x)^2}{4t}}y^mdy,
\end{equation*}
of the Green's function. Substituting $y\to y+\frac{x}{1+4i\varepsilon t}$ in this integral gives
\begin{align*}
g_m(t,x)&=\frac{1}{2\sqrt{i\pi t}}\lim\limits_{\varepsilon\to 0^+}e^{-\frac{\varepsilon x^2}{1+4i\varepsilon t}}\int_\mathbb{R}e^{-(\varepsilon-\frac{i}{4t})y^2}\Big(y+\frac{x}{1+4i\varepsilon t}\Big)^mdy \\
&=\frac{1}{2\sqrt{i\pi t}}\lim\limits_{\varepsilon\to 0^+}e^{-\frac{\varepsilon x^2}{1+4i\varepsilon t}}\sum\limits_{k=0}^m{m\choose k}\Big(\frac{x}{1+4i\varepsilon t}\Big)^{m-k}\int_\mathbb{R}e^{-(\varepsilon-\frac{i}{4t})y^2}y^kdy \\
&=\frac{1}{2\sqrt{i\pi t}}\sum\limits_{k=0}^m{m\choose k}x^{m-k}\lim\limits_{\varepsilon\to 0^+}\int_\mathbb{R}e^{-(\varepsilon-\frac{i}{4t})y^2}y^kdy.
\end{align*}
For the remaining integral we will now use the moments of the Gaussian
\begin{equation*}
\int_\mathbb{R}e^{-\lambda y^2}y^kdy=\begin{cases} 0, & \text{if }k\text{ is odd}, \\ \frac{k!\sqrt{\pi}}{(\frac{k}{2})!(4\lambda)^{\frac{k}{2}}\sqrt{\lambda}}, & \text{if }k\text{ is even}, \end{cases}\qquad\Re(\lambda)>0,
\end{equation*}
where for odd $k$ the integral vanishes due to the antisymmetric integrand and for even $k$ the value is given in \cite[Equation 7.4.4]{AS72}. Hence we get
\begin{align*}
g_m(t,x)&=\frac{1}{2\sqrt{i\pi t}}\sum\limits_{k=0,k\,{\rm even}}^m{m\choose k}x^{m-k}\lim\limits_{\varepsilon\to 0^+}\frac{k!\sqrt{\pi}}{(\frac{k}{2})!2^k(\varepsilon-\frac{i}{4t})^{\frac{k}{2}}\sqrt{\varepsilon-\frac{i}{4t}}} \\
&=m!\sum\limits_{k=0,k\,{\rm even}}^m\frac{x^{m-k}(it)^{\frac{k}{2}}}{(\frac{k}{2})!(m-k)!}\\
&=m!\sum\limits_{k=0}^{\lfloor\frac{m}{2}\rfloor}\frac{x^{m-2k}(it)^k}{k!(m-2k)!}.
\end{align*}
\end{exam}

\end{document}